\documentclass[12pt,a4paper]{article}
\usepackage{amsmath,amssymb,amsthm,amsfonts,amscd,euscript,verbatim, t1enc, newlfont}
\usepackage{hyperref}
\providecommand{\keywords}[1]{\textbf{\textit{Key words and phrases }} #1}
\providecommand{\subjclass}[1]{\textbf{\textit{2010 Mathematics Subject Classification.}} #1}

\theoremstyle{definition}
\newtheorem{theo}{Theorem}[subsection]

\newtheorem{pr}[theo]{Proposition}

 \newtheorem{coro}[theo]{Corollary}
  
\theoremstyle{remark}
\newtheorem{rema}[theo]{Remark}

\theoremstyle{definition}
\newtheorem{defi}[theo]{Definition}

\numberwithin{equation}{subsection}

 \newcommand\lan{\langle}
\newcommand\ra{\rangle}

\newcommand\ob{^{-1}}

\newcommand\dmeop{DM_{\zop}^{\operatorname{eff}}}

\newcommand\mg{M_{gm}}
\newcommand\schpr{\operatorname{SchPr}}

\newcommand\obj{\operatorname{Obj}}

\newcommand\id{\operatorname{id}}

\newcommand\cu{\underline{C}}
\newcommand\du{\underline{D}}

\newcommand\au{\underline{A}}
\newcommand\bu{\underline{B}}
\newcommand\hu{\underline{H}}
\newcommand\cp{\mathcal{P}}
\newcommand\perpp{{}^{\perp}}
\newcommand\opp{^{op}}
\newcommand\hf{{\underline{HF}}}
\newcommand\hrt{{\underline{Ht}}}
\newcommand\hw{{\underline{Hw}}}
\newcommand\hv{{\underline{Hv}}}

\newcommand\z{{\mathbb{Z}}}
\newcommand\zll{{\mathbb{Z}_{(\ell)}}}
\newcommand\zl{{\mathbb{Z}_{\ell}}}
\newcommand\ql{\mathbb{Q}_{\ell}}
\newcommand\znz{{\mathbb{Z}/{\ell}^n\z}}

\newcommand\zop{\mathbb{Z}[1/p]} 
\newcommand\com{\mathbb{C}}

\newcommand\q{{\mathbb{Q}}}
\newcommand\af{\mathbb{A}}
 
\newcommand\p{\mathbb{P}}
\newcommand\afo{\mathbb{A}^1} 
\newcommand\pt{\operatorname{pt}}
\newcommand\thomr{t_{\operatorname{hom}}^R}

\newcommand\al{\alpha}

\newcommand\ns{\{0\}}

\DeclareMathOperator\torr{\operatorname{Tor}}

\DeclareMathOperator\inli{\varinjlim}
\DeclareMathOperator\kar{\operatorname{Kar}}
\newcommand\codim{\operatorname{codim}}

\newcommand\chow{\operatorname{Chow}}

\newcommand\chower{\operatorname{Chow}^{\operatorname{eff}}_R}
\newcommand\chowez{\operatorname{Chow}^{\operatorname{eff}}_{\z}}
\newcommand\chowr{\operatorname{Chow}_R}
\newcommand\dmbir{DM_R^o{}}
\newcommand\dmgbir{DM_{\operatorname{gm},R}^o{}}
\newcommand\dmgbirp{DM_{\operatorname{gm},R}^{o'}}
\newcommand\grwd{Gr^{W}}
\newcommand\tst{t^{st}}

\newcommand\hir{\operatorname{HI}_R}
\newcommand\hiz{\operatorname{HI}_R^o}
\newcommand\for{\operatorname{For}}
\newcommand\ca{{\mathcal{A}}}
\newcommand\dmer{DM^{\operatorname{eff}}_{R}{}}
\newcommand\dmger{DM^{\operatorname{eff}}_{\operatorname{gm},R}{}}

\newcommand\dmgr{DM_{\operatorname{gm}}^R{}}
\newcommand\mgr{M_{\operatorname{gm}}^R}

\newcommand\mgb{M^o_{\operatorname{gm}}}
\newcommand\wgbir{w^o_{\operatorname{gm}}}
\newcommand\wgbirp{w^{o'}_{\operatorname{gm}}}
\newcommand\hwgbir{\hw^o_{\operatorname{gm}}}
\newcommand\mgcr{M_{\operatorname{gm}}^{c}}

\newcommand\tcho{t_{\chow}}

\newcommand\ab{Ab}
\newcommand\var{\operatorname{Var}}
\newcommand\sv{\operatorname{SmVar}}
\newcommand\sm{\underline{\operatorname{Sm}}}
\newcommand\spv{\operatorname{SmPrVar}}

\newcommand\spe{\operatorname{Spec}}

\DeclareMathOperator\imm{\operatorname{Im}}
\DeclareMathOperator\co{\operatorname{Cone}}

\DeclareMathOperator\cha{\operatorname{char}}

\newcommand\modd{\operatorname{mod}}
\newcommand\zlmod{\zl-\operatorname{mod}}
\newcommand\znzmod{{\mathbb{Z}/{\ell}^n\z}-\operatorname{mod}}

\newcommand\tr{\operatorname{tr}}
\newcommand\add{\operatorname{add}}
\newcommand\ii{i^0}
\newcommand\bra{\mathcal{B}r}
\newcommand\rgm{RG_m}

\newcommand\wstu{w_{st}}

\newcommand\cuw{\cu_w}
\newcommand\kw{K_{\mathfrak{w}}}

\newcommand\wchows{w_{\chow}^{eff}}

\newcommand\wchowsg{w_{\chow}^{eff,gm}}

\newcommand\wchow{w_{\chow}}
\newcommand\wchowg{w^{gm}_{\chow}}
\newcommand\wbir{w^{o}}

\newcommand\chowbir{\chow^{o}}

\newcommand\rnrz{R_{nr}^{0}}

\begin{document}

 \title{On  Chow-pure cohomology and  Euler characteristics  for motives and varieties, and their relation to unramified cohomology and Brauer groups} 
 \author{M.V. Bondarko
   \thanks{ 
 The main results of the paper were  obtained under support of the Russian Science Foundation grant no. 16-11-00200.
} }
 \maketitle
\begin{abstract}
 We  study Grothedieck groups of triangulated categories using weight structures, weight complexes, and the corresponding pure  (co)homological functors. We prove some general statements  on $K_0$ of weighted  categories and apply it to Voevodsky motives endowed with so-called Chow weight structures.  We obtain certain "motivic substitutes" for smooth compactifications of smooth varieties over arbitrary perfect fields; this enables us to make certain unramified cohomology and Euler characteristic calculations  that are closely related to  results of T. Ekedahl and B. Kahn.
\end{abstract}




\tableofcontents

 \section*{Introduction}

In this paper we discuss 
 {\it weighted} triangulated categories\footnote{These are the triangulated categories endowed with weight structures  as introduced 
 by the author and independently by D. Pauksztello.} and apply some general results to Voevodsky motives.

  In particular, we use the properties of {\it weight-exact} localizations 
	  to obtain that for any smooth variety $X$ there exists a Chow motif that is "birational" to it. This result is a trivial  consequence of the Hironarka's resolution of singularities if the characteristic $p$ of the base field is  zero, but it is 	 new in the case $p>0$ (we prove it for $\zop$-linear motives). 
	Applying certain results of \cite[\S7]{kabir} we deduce that this statement can be applied to unramified cohomology calculations. In particular, one may consider the $\zll$-component for the (cohomological) Brauer groups of smooth varieties, where $\ell$ is a prime distinct from $p$. 

Moreover, we prove that weight complex functors (as introduced in \cite{bws} and \cite{bwcp}; however, the term originates from \cite{gs})  can be used to calculate the Grothendieck groups of ({\it bounded}) weighted categories. The corresponding general Theorem \ref{teulerw} extends Theorem 5.3.1 of \cite{bws}. In particular, it says  that if a weight structure $w$ on $\cu$ is bounded then any additive functor $F$ from the heart $\hw$ of $w$ into a category $\au$ gives a homomorphism $ F_{K_0}: K_0^{\tr}(\cu)\to K_0^{\add}(\au)$ (see Definition \ref{dk0} below); one may call homomorphisms of this sort {\it Euler characteristic} ones. An important observation here that in the case where $\cu=\dmger$ (the category of $R$-linear geometric Voevodsky motives)  one can apply $ F_{K_0}$ to motives with compact support $\mgcr(-)$ 
  to obtain a function $G$ from the set of $k$-varieties into $K_0^{\add}(\au)$  that satisfies the scissors relation $$G(X)=G(X\setminus Z)+G(Z)$$ if $Z$ is a closed subvariety of $X$. 
	 In Theorem \ref{tceked} we take $F=E^i$, where $E^i$ is the obvious extension of $H^i(-,\zl)$ to $\chower$ (where $R$ is a localization of $\zop$) to obtain results closely related to Ekedahl's invariant calculations in \cite[Theorem 5.1]{eked} (see Remark \ref{reked}(\ref{irek1})).   

Now we describe the contents of the paper. More details can be found at the beginnings of sections.

In \S\ref{swsr} we recall several definitions and results from previous papers; they are mostly related to triangulated categories and weight structures. In particular, in \S\ref{stable} we recall some of the theory of (strong) weight complexes along with their relation to the so-called pure (co)homological functors and calculations in Grothendieck groups of triangulated categories; these results are 
 generalized in \S\ref{skz}.

In \S\ref{smotchow} we apply the results of previous sections to to the study of various categories of  motives (including the birational ones introduced  in \cite{kabir}), Chow weight structures on them, and their unramified cohomology. In particular, we study the (pure) unramified Brauer group functor  and compute $ E^i_{K_0}(\mgcr(X))$ for certain smooth $X/k$ and $E^i$ as above. We also 
  demonstrate that this case of our results is closely related to 
	 Theorem 5.1 of \cite{eked}.
 
In \S\ref{ssupl} we give some detail and generalizations for the results of the previous sections. In particular, Theorem \ref{teulerw}(II.1) calculates the Grothendieck group of an (arbitrary) bounded weighted category, and Proposition \ref{plocw} can be used to study weight-exact localizations of ("big") motivic categories.

The author is deeply grateful to Federico Scavia for calling his attention to the topic and to the preprint \cite{eked}, to Vladimir Sosnilo for telling  him Proposition \ref{plocw}(\ref{iwl312}), 
 and to Bruno Kahn for his very interesting comments. 

\section{On weight structures in triangulated categories: reminder}\label{swsr}

In this section we recall those parts of the theory of weight structures that will be applied below. 

In \S\ref{snotata} we introduce some definitions and notation for (mostly, triangulated) categories.

In \S\ref{ssws} we give the basic definitions and properties of weight structures; this includes a new statement on weight-exact localizations. 

In \S\ref{stable} we recall some of the theory of (strong) weight complexes along with their relation to the so-called pure (co)homological functors and calculations in Grothendieck groups of triangulated categories.


\subsection{Some definitions 
for (triangulated) categories}\label{snotata} 
\begin{itemize}

\item All products and coproducts in this paper will be small.

\item Given a category $C$ and  $X,Y\in\obj C$  we will write
$C(X,Y)$ for  the set of morphisms from $X$ to $Y$ in $C$.

\item For categories $C'$ and $C$ we write $C'\subset C$ if $C'$ is a full 
subcategory of $C$.

\item Given a category $C$ and  $X,Y\in\obj C$, we say that $X$ is a {\it
retract} of $Y$ 
 if $\id_X$ can be 
 factored through $Y$.\footnote{Clearly,  if $C$ is triangulated or abelian, 
then $X$ is a retract of $Y$ if and only if $X$ is its direct summand.}\ 

\item Below $\bu$ will always  denote some additive category.

\item A 
 subcategory $\hu$ of   $\bu$ 
is said to be {\it retraction-closed} in $\bu$ if it contains all $\bu$-retracts of its objects. 

Moreover, for  a class $D$ of objects of $\bu$ we use the notation $\kar_{\bu}(D)$ for the class of all $\bu$-retracts of elements of $D$.

\item We will say that 
 $\bu$ is {\it idempotent complete} if any idempotent endomorphism gives a direct sum decomposition in it.

\item The {\it idempotent completion} $\kar(\bu)$ (no lower index) of 
 $\bu$ is the category of ``formal images'' of idempotents in $\bu$.
Respectively, its objects are the pairs $(B,p)$ for $B\in \obj \bu,\ p\in \bu(B,B),\ p^2=p$, and the morphisms are given by the formula 
$$\kar(\bu)((X,p),(X',p'))=\{f\in \bu(X,X'):\ p'\circ f=f \circ p=f \}.$$ 
 The correspondence  $B\mapsto (B,\id_B)$ (for $B\in \obj \bu$) fully embeds $\bu$ into $\kar(\bu)$, and it is well known that $\kar(\bu)$ is essentially the smallest idempotent complete additive category containing $\bu$.

\item We will use the term "exact functor" for a functor of
triangulated categories (i.e.  for a functor that preserves the
structures of triangulated categories). 

\item The symbol $\cu$ below will always denote some triangulated category;  it will often be endowed with a weight structure $w$ (which usually will be bounded). The symbols $\cu'$ and $\du$ will  also be used  for triangulated categories only.

\item For any  $A,B,C \in \obj\cu$ we will say that $C$ is an {\it extension} of $B$ by $A$ if there exists a distinguished triangle $A \to C \to B \to A[1]$.

 A class $\cp\subset \obj \cu$ is said to be  {\it extension-closed}
    if it 		is closed with respect to extensions and contains $0$.  

\item 
We will write $\lan \cp\ra$ for the smallest  full retraction-closed triangulated subcategory of $\cu$ containing $\cp$; we will  call  $\lan \cp\ra$  the triangulated subcategory {\it densely generated} by $\cp$ (in particular, in the case $\cu=\lan \cp \ra$).\footnote{Alternatively, $\lan \cp\ra$  can be called the thick subcategory of $\cu$ generated by $\cp$.} Moreover, if $\hu$ is the full subcategory of $\cu$ such that $\obj \hu=\cp$ then we will also say that $\cu$ is densely generated by $\hu$.

\item 
The  smallest  {\bf strictly}  full triangulated subcategory of $\cu$ containing $\cp$ will be called the subcategory {\it strongly generated} by $\cp$ and also by $\hu$. 


\item For $X,Y\in \obj \cu$ we will write $X\perp Y$ if $\cu(X,Y)=\ns$. 

For
$D,E\subset \obj \cu$ we write $D\perp E$ if $X\perp Y$ for all $X\in D,\
Y\in E$.

Given $D\subset\obj \cu$ we  will write $D^\perp$ for the class
$$\{Y\in \obj \cu:\ X\perp Y\ \forall X\in D\}.$$
Dually, ${}^\perp{}D$ is the class $\{Y\in \obj \cu:\ Y\perp X\ \forall X\in D\}$.


	\item We will say that an additive covariant (resp. contravariant) functor from $\cu$ into an abelian category $\au$ is {\it homological} (resp. {\it cohomological}) if it converts distinguished triangles into long exact sequences. 
	
	For a (co)homological functor $H$ and $i\in\z$ we will write $H_i$ (resp. $H^i$) for the composition $H\circ [-i]$. 

\item 
We will write $K(\bu)$ for the homotopy category of (cohomological) complexes over $\bu$. Its full subcategory of
bounded complexes will be denoted by $K^b(\bu)$. We will write $M=(M^i)$ if $M^i$ are the terms of the complex $M$.
\end{itemize}

\subsection{Weight structures: basics}\label{ssws}

Let us recall some basic  definitions of the theory of weight structures. 

\begin{defi}\label{dwstr}

A pair of subclasses $\cu_{w\le 0},\cu_{w\ge 0}\subset\obj \cu$ 
will be said to define a {\it weight structure} $w$ on a triangulated category  $\cu$ if 
they  satisfy the following conditions.

(i) $\cu_{w\ge 0}$ and $\cu_{w\le 0}$ are 
retraction-closed in $\cu$ (i.e., contain all $\cu$-retracts of their objects).

(ii) {\bf Semi-invariance with respect to translations.}

$\cu_{w\le 0}\subset \cu_{w\le 0}[1]$, $\cu_{w\ge 0}[1]\subset
\cu_{w\ge 0}$.

(iii) {\bf Orthogonality.}

$\cu_{w\le 0}\perp \cu_{w\ge 0}[1]$.

(iv) {\bf Weight decompositions}.

 For any $M\in\obj \cu$ there
exists a distinguished triangle
$$LM\to M\to RM {\to} LM[1]$$
such that $LM\in \cu_{w\le 0} $ and $ RM\in \cu_{w\ge 0}[1]$.

Moreover, we will say that a triangulated category $\cu$ is {\it weighted} if it is endowed with a fixed weight structure.  
\end{defi}

We will also need the following definitions.

\begin{defi}\label{dwso}
Let $i,j\in \z$; assume that a triangulated category $\cu$ is endowed with a weight structure $w$.

\begin{enumerate}
\item\label{idh}
The full  subcategory $\hw$ of $ \cu$ whose objects are
$\cu_{w=0}=\cu_{w\ge 0}\cap \cu_{w\le 0}$ 
 is called the {\it heart} of 
$w$.

\item\label{id=i}
 $\cu_{w\ge i}$ (resp. $\cu_{w\le i}$, resp. $\cu_{w= i}$) will denote the class $\cu_{w\ge 0}[i]$ (resp. $\cu_{w\le 0}[i]$, resp. $\cu_{w= 0}[i]$).

\item\label{idbob} We will call $\cup_{i\in \z} \cu_{w\ge i}$ (resp. $\cup_{i\in \z} \cu_{w\le i}$) the class of {\it $w$-bounded below} (resp., {\it $w$-bounded above}) objects of $\cu$; we will write $\cu^b$ for the full subcategory of $\cu$ whose objects are bounded both above and below. 



\item\label{idbo} We will  say that $(\cu,w)$ is {\it  bounded} and $\cu$ is a {\it bounded weighted category}  if $\cu^b=\cu$,  i.e., if $\cup_{i\in \z} \cu_{w\le i}=\obj \cu=\cup_{i\in \z} \cu_{w\ge i}$.

\item\label{idwe}  Let 
  $\cu'$ be a triangulated category endowed with  a weight structure $w'$; let $F:\cu\to \cu'$ be an exact functor.

Then $F$ is said to be  {\it  weight-exact} (with respect to $w,w'$) if it maps $\cu_{w\le 0}$ into $\cu'_{w'\le 0}$ and
sends $\cu_{w\ge 0}$ into $\cu'_{w'\ge 0}$. 

\item\label{idrest}
Let $\du$ be a full triangulated subcategory of $\cu$.

We will say that $w$ {\it restricts} to $\du$ whenever the couple $w_{\du}=(\cu_{w\le 0}\cap \obj \du,\ \cu_{w\ge 0}\cap \obj \du)$ is a weight structure on $\du$. If this is the case then we will call $(\cu_{w\le 0}\cap \obj \du,\ \cu_{w\ge 0}\cap \obj \du)$ the {\it restriction} of $w$ to $\du$, and say that the weight structure $w_{\du}$ {\it extends} to $\cu$.

\item\label{idneg}
 We will say that the subcategory 
 $\hu$  is {\it negative} (in $\cu$) if $\obj \hu\perp (\cup_{i>0}\obj (\hu[i]))$.\footnote{In several papers (mostly, on representation theory and related matters) a negative subcategory 
satisfying certain additional assumptions was said to be {\it silting}; this notion generalizes the one of {\it tilting}}

\item\label{idwic} We will say that an additive category $\bu$ is {\it weakly idempotent complete}\footnote{It appears that this term was introduced in \cite[Definition 7.2]{buhler}, whereas the notion itself probably originates from  \cite{freydsplit}, where it was said that 
{\it retracts have complements} in $\bu$.} 
if any split monomorphism $i:X\to Y$ (that is, $\id_X$ equals $p\circ i$ for some $g\in \bu(Y,X)$)  is isomorphic to the 
	 monomorphism $\id_X\bigoplus 0:X\to X\bigoplus Z$ for some object  $Z$ of $\bu$. 
\end{enumerate}
\end{defi}

\begin{rema}\label{rstws}
1. A  simple (and still  quite useful) example of a weight structure comes from the stupid filtration on the homotopy category of cohomological complexes $K(\bu)$ for an arbitrary additive  $\bu$; it can also be restricted to the subcategory $K^b(\bu)$ of bounded complexes (see Definition \ref{dwso}(\ref{idrest})). In this case $K(\bu)_{\wstu\le 0}$ (resp. $K(\bu)_{\wstu\ge 0}$) is the class of complexes that are homotopy equivalent to complexes  concentrated in degrees $\ge 0$ (resp. $\le 0$); see Remark 1.2.3(1) of \cite{bonspkar} for more detail. 

The heart of the weight structure $\wstu$ is the retraction-closure  of $\bu$  in  $K(\bu)$; hence it is equivalent to $\kar(\bu)$ (since both $K^{-}(\bu)$ and $K^+(\bu)$ are idempotent complete).

2. In the current paper we use the ``homological convention'' for weight structures; 
it was previously used in  
\cite{wild}, 
\cite{bonspkar}, \cite{binters}, 
and in \cite{bokum}, 
whereas in \cite{bws} and \cite{bger}, 
 the ``cohomological convention'' was used. In the latter convention  the roles of $\cu_{w\le 0}$ and $\cu_{w\ge 0}$ are essentially interchanged, that is, one considers   $\cu^{w\le 0}=\cu_{w\ge 0}$ and $\cu^{w\ge 0}=\cu_{w\le 0}$. 
 
 We also recall that D. Pauksztello has introduced weight structures independently in \cite{konk}; he called them
co-t-structures. 

3. Till section \ref{ssupl} 
 the reader may assume that all weight structures we consider are bounded. Note that weight structures of this sort have an easy description in terms of negative subcategories; see Proposition \ref{pbw}(\ref{ineg}, \ref{igen}, \ref{igenlm}).
\end{rema}

\begin{pr}\label{pbw}
Let  
$m\le n\in\z$, $M,M'\in \obj \cu$, $g\in \cu(M,M')$. 

\begin{enumerate}
\item \label{idual}
The axiomatics of weight structures is self-dual, i.e., for $\cu'=\cu^{op}$
(so $\obj\cu'=\obj\cu$) there exists the (opposite)  weight structure $w'$ for which $\cu'_{w'\le 0}=\cu_{w\ge 0}$ and $\cu'_{w'\ge 0}=\cu_{w\le 0}$.

\item\label{iext}  $\cu_{w\le 0}$, $\cu_{w\ge 0}$, and $\cu_{w=0}$ are (additive and) extension-closed.

\item\label{igenlm}
 $\cu^b$ is the triangulated subcategory of  $\cu$ strongly generated by $\cu_{w=0}$.



\item\label{iort}  $\cu_{w\ge 0}=(\cu_{w\le -1})^{\perp}$ and $\cu_{w\le 0}={}^{\perp} \cu_{w\ge 1}$.


\item\label{itstable} $\hw$ is weakly idempotent complete. 

\item\label{ineg} Any full subcategory of $\hw$ is a negative subcategory of $\cu$.

\item\label{igen} Assume that $\cu$ is densely generated by its negative additive subcategory $\bu$. 

Then there exists a unique weight structure $w_{\bu}$ on $\cu$ whose heart contains $\bu$. 
Moreover, this weight structure is bounded, $\cu_{w_{\bu}=0}=\kar_{\cu}(\obj \bu)$, and $\cu_{w_{\bu}\ge 0}$ (resp. $\cu_{w_{\bu}\le 0}$) is the smallest class of objects that contains $\obj\bu[i]$ for all $i\ge 0$ (resp. $i\le 0$), and is also  extension-closed and retraction-closed in $\cu$.

\item\label{iwex} Let   $\cu'$ be a triangulated category endowed with  a weight structure $w'$; let $F:\cu\to \cu'$ be an exact functor.

If $F$ is weight-exact then it sends $\cu_{w=0}$ into $\cu'_{w'=0}$. Moreover, if $w$ is bounded then the converse implication is valid is well. 
\end{enumerate}
\end{pr}
\begin{proof}
 Assertions \ref{idual}--\ref{itstable} were essentially established  in \cite{bws} (yet cf.  Remark 1.2.3(4) of \cite{bonspkar}, pay attention to Remark \ref{rstws}(2) above,  and see \cite{bvheart} for some more detail on assertion \ref{itstable}).

Assertion \ref{ineg} follows from the orthogonality axiom (iii) 
 in Definition \ref{dwstr} immediately. 

Assertion \ref{igen} is given by Corollary 2.1.2 of \cite{bonspkar}.

The first implication in assertion \ref{iwex} is obvious. The converse implication easily follows from assertion \ref{igenlm}; see also Lemma 2.7.5 of \cite{bger}.
\end{proof}


Let us now recall some statements on weight-exact localizations; more information can be found in Proposition \ref{plocw} and Remark \ref{rsos}(3) below. 

\begin{pr}\label{ploc}
 Assume that $\cu$ is endowed with a bounded weight structure $w$,  $\du\subset \cu$ is a triangulated subcategory 
  strongly generated by a class $\cp\subset \cu_{w=0}$,  and $\pi$ is the localization functor $\pi:\cu\to \cu'=\cu/\du$.   

\begin{enumerate}

\item\label{iwl31} 
 Then $w$ restricts to $\du$, and there exists a bounded weight structure $w'$ on $\cu'$ such that $\pi$ is weight-exact and  the class $\cu'_{w'= 0}$ essentially equals $\pi(\cu_{w= 0})$.


\item\label{iwl4}
Moreover, for any $M\in \cu_{w\le 0}$ and  $N\in \cu_{w\ge 0}$
the homomorphism $\cu(M,N)\to \cu'(\pi(M),\pi(N))$ 
 is surjective. 

Consequently, if  
$\pi(M)\cong \pi(N')$ for some $N'\in \cu_{w=0}$ then there exists $f\in \cu(M,N')$ such that 
 $\pi(f)$ is an isomorphism.
\end{enumerate}
\end{pr}
\begin{proof}
\ref{iwl31}. Proposition \ref{pbw}(\ref{ineg}) implies that $w$ restricts to $\du$ indeed. By Proposition 3.3.3(1) of   \cite{bsosnl} 
 there exists a weight structure $w'$ on $\cu$ such that $\pi$ is weight-exact, and the class $\cu'_{w'=0}$ essentially equals $\pi(\cu_{w= 0})$. Lastly, Proposition \ref{pbw}(\ref{igenlm}) implies that $w'$ is bounded; see also Proposition 8.1.1(3) of \cite{bws}.

\ref{iwl4}.  Proposition 1.8(I.3) of \cite{binters} 
 gives the first part of the assertion.

Next, if $\pi(M)\cong \pi(N')$ for  $N'\in \cu_{w=0}$ then $N'$ also belongs to $\cu_{w\ge 0}$. Thus we can apply the first part 
 for $N'=N$ to lift the isomorphism $\pi(M)\to \pi(N')$ to a $\cu$-morphism $f$. 
\end{proof}

\subsection{On (strong) weight complexes, pure functors, and  Grothendieck groups of 
  weighted categories}\label{stable}

Now we recall the theory of so-called "strong" weight complex functors. Note here that this version of the theory is less general than the "weak" one that we will recall in 
 \S\ref{sbr} below. 
 However, both versions of the theory are fine for the  purposes of the current paper, and we start from the strong one that it easier to describe.

\begin{pr}\label{pwt}
 Assume that $\cu$ possesses an $\infty$-enhancement (see \S1.1 of \cite{sosnwc} for the corresponding references), and is endowed with a bounded weight structure $w$.

Then there exists an exact functor $t^{st}:\cu\to K^b(\hw)$, $M\mapsto (M^i),$ that enjoys the following properties. 

\begin{enumerate}
\item\label{iwcbase} The composition of the embedding $\hw\to \cu$ with $t^{st}$ is isomorphic to the obvious embedding $\hw \to K^b(\hw)$.


 \item\label{iwcfunct} Let $\cu'$ be a triangulated category that possesses an $\infty$-enhancement as well and is endowed with a bounded weight structure $w'$; let  $F:\cu\to \cu'$ be a weight-exact functor that lifts to $\infty$-enhancements. Then the composition $t'{}^{st}\circ F$ is isomorphic to $K^b(\hf)\circ t{}^{st}$, where 
$t'{}^{st}$ is the weight complex functor corresponding 
 $w'$, and the functor $K^b(\hf):K^b(\hw)\to K^b(\hw')$ is the obvious $K^b(-)$-version of the restriction $\hf:\hw\to \hw'$.

Moreover, the obvious modification of this statement 
 corresponding to contravariant weight-exact functors (cf. Proposition \ref{pbw}(\ref{idual})) is valid as well.


\item\label{iwcons} If $M\in \cu_{w\le n}$ (resp. $M\in \cu_{w\ge n}$) then $t^{st}(M)$ belongs to $K(\hw)_{\wstu\le n}$ (resp. to $K(\hw)_{\wstu\ge n}$).

\item\label{iwcpu} Assume that $\ca$ is an additive covariant (resp., contravariant) functor from $\hw$ into an abelian category $\au$. Then the functor $H^{\ca}$ (resp. $H_{\ca}$) that sends $M$ into the zeroth homology of the complex $\ca(M^i)$ (resp. $\ca(M^{-i})$)  is (co)homological. Moreover, this is the only  (co)homological functor (up to a unique isomorphism) whose restriction to $\hw$  equals $\ca$ and whose restrictions to $\hw[i]$ for $i\neq 0$ vanish. 
\end{enumerate}
\end{pr}
\begin{proof}

Assertions \ref{iwcbase} and \ref{iwcfunct} easily  follow from  Corollary 3.5 of \cite{sosnwc} (along with its proof which is essentially self-dual); 
  see also  \S6.3 of \cite{bws} for the case where $\cu$ possesses a differential graded enhancement (and that is sufficient for our purposes below). 
	
	The easy assertion \ref{iwcons} is given by Proposition 1.3.4(10) of \cite{bwcp} (note however that to apply the results of ibid. one should recall that $\tst$ is compatible with the {\it weak} weight complex functor as defined in loc. cit.; see Remark 1.3.5(3) of ibid. and Corollary \cite{sosnwc}  for this statement, whereas the "weak" theory itself is recalled in Proposition \ref{pwtw} below).
	
	The only non-trivial statement in assertion \ref{iwcpu} is the uniqueness one, that is given by Theorem 2.1.2(2,3)  of  \cite{bwcp}.
	\end{proof}

\begin{rema}\label{rwc}
The term "weight complex" originates from \cite{gs}; yet the domains of the  ("concrete") weight complex functors considered in that paper were not triangulated. 
\end{rema}


Now let us relate weight complexes to certain Grothendieck groups.

\begin{defi}\label{dk0}
Let $\bu$ be an essentially small additive category, whereas $\cu$ is an essentially small triangulated category.

1. Then the (split) Grothendieck group $K_0^{\add}(\bu)$ of  $\bu$ is the abelian group whose generators are 
the isomorphism classes of objects of $\bu$, and the relations are of the form $[B]=[A]+[C]$ for all   $A, B,C\in\obj \bu$ such that $B\cong A\bigoplus C$. 

2. 
 The Grothendieck group $K_0^{\tr}(\cu)$  is the abelian group whose generators are the isomorphism classes of objects of $\cu$, and the relations are of the form $[B]=[A]+[C]$, where  $A\to B\to C\to B[1]$ is a $\cu$-distinguished triangle. 
\end{defi}

Now we will formulate a rather general statement on Grothendieck groups; we will generalize it in Theorem \ref{teulerw} below.

\begin{theo}\label{teuler}
Adopt the assumptions of Proposition \ref{pwt}; assume 
$N$ is an object of $\cu$ and $F:\hw\to \au$ is an additive functor (consequently, $\au$ is additive as well).

 1. Then the correspondence $M\mapsto \sum_{i\in \z} (-1)^i[F(M^i)]$ (where $t^{st}(M)=(M^i)$) gives a well-defined homomorphism $F_{K_0}:K_0^{\tr}(\cu)\to K_0^{\add}(\au)$. 

2. Assume that 
  $\au$ is an abelian category, and there exists a $\cu$-morphism $h$ either from $N'$ to  $N$ or vice versa such that $F_{K_0}([\co(h)])=0$ and $N'\in \cu_{w=0}$. Then $F_{K_0}([N])=[F(N')]$. 

Moreover, if $H_j^F(\co(h))=0$ for 
$j=0,1$ and the homological functor $H^F$ given by  Proposition \ref{pwt}(\ref{iwcpu}), then 
$F_{K_0}([N])=[H^{F}(N)]$. 

3. Assume that $\au$ is abelian semi-simple. Then $F_{K_0}(N)=\sum_{i\in \z} (-1)^i[H^F_i(N)]$.

4. Assume that $F=\id_{\hw}$. Then the corresponding homomorphism $\id_{\hw,K_0}: K_0^{\tr}(\cu)\to K_0^{\add}(\hw)$ is an isomorphism.
\end{theo}
\begin{proof}
1. Obviously $t\circ K^b(F)$ gives a well-defined homomorphism $K_0^{\tr}(\cu)\to K_0^{\tr}(K^b(\au))$. It remains to recall that 
$K_0^{\tr}(K^b(\au))\cong K_0^{\add}(\au)$ by Theorem 1 of \cite{rose}; see also 
 Theorem \ref{teulerw}(II.1) and Remark \ref{reulerw}(1) below.

2. Immediately  from the definition of  $K_0^{\tr}(\cu')$, we have $F_{K_0}([N])=F_{K_0}([N'])$. Next, 
  assertion 1 implies that $F_{K_0}([N])=[H^F(N')]=F_{K_0}([N'])$.

Lastly, if $H^F(\co(h))=0=H^F(\co(h)[-1])$ then we have 
 $H^F(N')=H^F(N)$,  since the functor $H^F$ is homological.

3. This statement is an immediate consequence of our definitions. 

4.  
 The embedding $\hw\to \cu$ obviously gives a homomorphism  $K_0^{\add}(\bu)\to K_0^{\tr}(\cu)$; it is surjective since $\hw$ strongly generates $\cu$ according to Proposition \ref{pbw}(\ref{igenlm}).  Hence $\id_{\hw,K_0}$ is an isomorphism, since its  composition with 
  $F_{\id}$ obviously equals $\id_{K_0^{\add}(\bu)}$. 
\end{proof}

\begin{rema}\label{reuler}
Assume that we have a "tensor product" bi-functor $\nu:\cu\times \cu\to \cu$ that becomes and exact functor if one of the arguments is fixed. It clearly defines a bi-additive operation $K_0^{\tr}(\cu)\times K_0^{\tr}(\cu)\to K_0^{\tr}(\cu)$. 

Next, if $\mu$ restricts to a bi-functor $\hw\times \hw\to \hw$, and the latter is compatible (via $F$) to a certain bi-additive functor $\au\times\au\to \au$ then the homomorphism $F_{K_0}$ clearly becomes a ring one (with respect to the corresponding operations).

2. The idea for parts 2 and  3 
 of our theorem is that in some cases the value of 	$F_{K_0}([M])$ can be computed by means of the homological functor $H^{F}$. We will discuss concrete examples of these calculations in  Theorem \ref{tceked} and 
  Remark \ref{reked}(\ref{iur}) below. 
\end{rema}

\section{Applications to motives and unramified cohomology}\label{smotchow} 

 Now we apply the results of previous sections to motives, Chow weight structures, and  unramified cohomology. 

In \S\ref{sbmot}  we recall some facts on  Voevodsky's effective motivic categories $\dmger\subset \dmer$ 
 and sheaves with transfers. 

In \S\ref{smotw} we 
 define the weight structure $\wchowsg$ on $\dmger$ and relate the corresponding pure functors to unramified cohomology. In particular, we study the unramified Brauer group sheaf with transfers. 

In \S\ref{smgc} we recall some properties of motives with compact support that demonstrate that applying functors of the sort $F_{K_0}$ (see Theorem \ref{teuler}(2)) to $\mgcr(-)$ 
 yields  (Euler characteristic) homomorphisms from 
  the Grothendieck group of varieties. 
 Next we apply our results in the case where $F$ is given by $\zll$-adic \'etale cohomology. 
 These statements are closely related to Theorem 5.1 of \cite{eked}.

\subsection{
On Voevodsky motives 
and sheaves with transfers}
\label{sbmot}

First we introduce some notation and recall some basics on Voevodsky motives.

 $k$ will be our perfect base field of characteristic $p$  ($p$ may be zero). 
 We introduce the following 
 convention: in the case $p=0$ the notation $\zop$ will 
 mean just the ring $\z$.

We will write $R$ for the coefficient ring of our motivic categories; respectively, $R$ is a unital commutative associative ring. The reader may assume that $R=\zop$ since this case is the most interesting one for the results of this paper.
On the other hand, when we will write $1/p\notin R$ we will always assume (in addition) that $p>0$.

 $\var\supset \sv\supset \spv$ will denote the set of all  (not necessarily connected) varieties over $k$, resp. of smooth varieties, resp. of smooth projective varieties. 
The category of smooth $k$-varieties will be denoted by $\sm$.


We will not define any categories of motives in this paper.  The reader can find these definitions in \cite{bev} and \cite{deg}; cf. also  \cite{1},  \cite{vbook}, and \cite{bokum}. 
 Instead,  we list some 
 well-known properties of motives (that mostly originate from \cite{1})  and introduce some notation.

\begin{pr}\label{prmot}
Let $X,Y\in \sv$.

\begin{enumerate}
\item\label{imot1}
The category $\dmer$ of (unbounded) $R$-linear motivic complexes is triangulated, $R$-linear, and  equipped with a functor $\mgr$ (of the $R$-linear motif) from $\sm$ into it. Moreover, $\mgr(\afo\times X\to X)$ is an isomorphism. 

We will write $\dmger$ for the subcategory of $\dmer$ densely generated by $\mgr(\sv)$; it is essentially small.  Moreover, we will write $\sm'$ for the full subcategory  of $\dmger\subset \dmer$ whose object class is $\mgr(\sv)$.

\item\label{imot3} $\dmger$ is a symmetric monoidal triangulated category. Moreover, the functor $\mgr$ converts products of varieties into tensor products, and sends disjoint unions into direct sums


\item\label{imot4} Let us write $\pt$ for $\spe k$, $R$ for $\mgr(\pt)$, and $R\lan 1 \ra$ for the "complement" of $R$ to $\mgr(\p^1)$ corresponding to the morphisms $\pt\to \p^1\to \pt$ (here we send $\pt$ into $0\in \p^1(k)$).

Then the functor $-\lan 1\ra=-\otimes R \lan 1\ra$ is  fully faithful on $\dmger$ (and on $\dmer$ as well). 

\item\label{imot5} Let $Z\in \sv$  be an equicodimensional closed subvariety of codimension  $n$ in 
 $X$  
(i.e.,  all of the connected components of $Z$ are of the same codimension  $n\ge 0$ in  $X$).
Then there exists a Gysin distinguished triangle 
$$\mgr(X\setminus Z) \to \mgr(X) \to\mgr(Z)\lan n \ra \to \mgr(X\setminus Z)[1]$$
in $\dmger$, where $-\lan n\ra$ is the $n$th iteration of $-\lan 1\ra$.

\item\label{imot6} Let us write $\hir$ for the {\it heart of the homotopy $t$-structure $\thomr$ on $\dmer$} (see \S\ref{sbr} below); consequently, $\hir$ is a full abelian subcategory of $\dmer$. We will call  the   objects of $\hir$ (Nisnevich homotopy invariant $R$-liear) {\it sheaves with transfers}.

 Then the correspondence sending an 
 a sheaf with transfers $N$ into the functor $\dmer (\mgr(-),N):\sm\opp\to \ab$ 
  gives 
	 a faithful exact functor $\for$ from $\hir$ into the abelian category of Nisnevich sheaves on $\sm$ (however, it will not probably cause any misunderstanding if one will write $N(X)$ instead of $\dmer (\mgr(X),N) $).

\item\label{imot9} Furthermore,	for any $N\in \obj \hir$ and $i\in \z$ the group $\dmer(\mgr(X),N[i])$ is naturally isomorphic to  $ H^i_{Nis}(X,\for(N))$ (i.e., to the $i$th Nisnevich 
 cohomology of  $X$ with coefficients in $\for(N)$; this is certainly zero if $i<0$).\footnote{Actually, these Nisnevich cohomology groups are canonically isomorphic to the corresponding Zariski ones.}

\item\label{imot9p}
Moreover, for any dense embedding $f$ of smooth $k$-varieties the homomorphism $\for(N)(f)$ is injective.

\item\label{imot7} The subcategory $\chower$ whose objects are the $\dmer$-retracts of elements of $\mgr(\spv)$ is negative in $\dmger$, and it is naturally equivalent to the category of $R$-linear effective Chow motives.

Moreover, if $R$ is a $\zop$-algebra then  $\chower$  densely generates $\dmger$ (cf. Theorem \ref{tmotw}(\ref{imotp1}) below). 

\item\label{imotet} Assume that 
$R$ is a localization of $\zop$. 
 Then there exists an object $\rgm$ of $\dmer$ such that the functors $\dmer(\mgr(-),\rgm[i])$ give 
 the \'etale cohomology $H^i_{et}(-,G_m\otimes R)$  (here $G_m$ is the unit group sheaf) for all $i\ge 0$. Moreover,   for  $S^{\bra}=H_{-2}^{\thomr}(\rgm)=H_0^{\thomr}(\rgm[2])$ (see Definition \ref{dtstrh} below) the restriction of $\dmer(-,S^{\bra})$ to the category $\sm'\subset \dmer$ (see assertion \ref{imot1}) 
 is isomorphic to the same restriction of $\dmer(\mgr(-),\rgm[2])$, and
 the sheaf $\for(S^{\bra})$ 
 is the  {\it $R$-linear cohomological Brauer group} functor $\bra:X\mapsto H^2_{et}(X,G_m\otimes R)$. 

\item\label{imotr} Assume that $f: Y\to X$ is a finite flat 
  $\sm$-morphism. Then the morphism $\operatorname{deg}(f)\cdot\id_{\mgr(X)}$ factors through the motif $\mgr(Y)$. 
 \end{enumerate}
 \end{pr}
\begin{proof}
 \S2.3, \S4.4,  Theorem 3.3,  and  Proposition 6.3.1  of \cite{bev}, and \S1.3 and Lemma 1.1.1 of \cite{bokum}
give assertions \ref{imot1}--\ref{imot9}.

 Next, assertion \ref{imot9p}  immediately follows from the existence of Gersten resolutions for objects of $\hir$; see Theorem  24.11 of \cite{vbook}.

The first part of assertion \ref{imot7}  immediately follows from  Corollary 6.7.3  of \cite{bev}.
Next, the category $\chower$ densely generates $\dmger$ if $R$ is a $\zop$-algebra since this statement is true in the case $R=\zop$ (here one can apply Proposition 1.3.3 of \cite{bokum}), whereas in this case the statement is given by Proposition 5.3.3 of \cite{kellyast} and also by Theorem 2.2.1(1) of  \cite{bzp} (cf. 
 Corollary 3.5.5 of \cite{1}). 

\ref{imotet}. The existence of $\rgm$  is well-known  (as well) and easily follows from 
 Example 6.22 of \cite{vbook}; see also the introduction of \cite{kahnr}. 

The calculation of $S^{\bra}$ is mentioned in Lemma 5.2 of ibid. We will give some detail for the proof of this statement and compare the restrictions of $\dmer(-,S^{\bra})$ and $\dmer(\mgr(-),\rgm[2])$ to $\sm'$ in \S\ref{sbr} below.

Lastly, assertion \ref{imotr} easily follows from Lemma 2.3.5 of \cite{2}; cf. the proof  of \cite[Corollary 1.2.2(1)]{bzp}.\footnote{
I am deeply grateful to  prof. D.-Ch. Cisinski for teaching me this argument.}
 \end{proof}

\subsection{On Chow weight structures and unramified cohomology}\label{smotw}

Starting from this moment we will assume that $R$ is a $\zop$-algebra.

\begin{theo}\label{tmotw}

Assume that  $X\in \sv$. Then the following statements are valid.

\begin{enumerate}

\item\label{imotp1} The subcategory $\chower$ of $\dmger$ actually strongly generates $\dmger$. Moreover, there exists  a bounded weight structure $\wchowsg$ on $\dmger$ whose heart equals $\chower$, and  $\mgr(X)\in \dmger_{\wchowsg\le 0}$.

Furthermore, 
 if  $X$ is of dimension $n>0$ then $\mgr(X)$ belongs to the subcategory of $\dmger$ densely generated by motives of projective varieties of dimension at most $n$. 

\item\label{imotw2} 
 The twist functor $-\lan 1\ra$ is weight-exact with respect to $\wchowsg$. Consequently, $\wchowsg$ restricts to the subcategory $\dmger\lan 1\ra$ (see Proposition \ref{prmot}(\ref{imot4})),  and there exists  a weight structure $\wgbirp$ on  
 the localization $ \dmgbirp=\dmger/\dmger\lan 1 \ra$ such that the localization functor $\pi:\dmer\to \dmgbirp$ is weight-exact.  

\item\label{imotw3}  The weight structure $\wgbirp$  extends (see Definition \ref{dwso}(\ref{idrest})) to a bounded weight structure $\wgbir$ on  the category $\dmgbir=\kar(\dmgbirp)$ whose heart $\hwgbir=\chowbir$  
  consists of the retracts of elements of $\mgb(\sv)$, where  $\mgb=\pi\circ \mgr$.

\item\label{imotp2} If $X\in \sv$ then there exists $M\in \obj \chower=\dmger{}_{\wchowsg=0}$ and $h\in \dmger(\mgr(X),M)$ such that $\pi(h)$ is an isomorphism.

\end{enumerate}
 \end{theo}
\begin{proof}
\ref{imotp1}. Applying Proposition \ref{pbw}(\ref{igen}) along with the negativity of $\chower$ in $\dmger$ provided by Proposition \ref{prmot}(\ref{imot7}) we obtain the existence of a bounded weight structure  $\wchowsg$ on $\dmger$ whose heart equals $\chower$; note that $\chower$  is retraction-closed in $\dmger$ by definition.  Applying Proposition  \ref{pbw}(\ref{igenlm}) we also obtain that $\chower$ strongly generates $\dmger$. Next, Proposition 6.7.3 of \cite{bev} implies that $\mgr(X)\perp \dmger_{\wchowsg\ge 1}$; hence $\mgr(X)\in \dmger_{\wchowsg\le 0}$ by Proposition \ref{pbw}(\ref{iort}) (alternatively, see \S2.1 of \cite{bokum}). 

Lastly, the aforementioned arguments of \cite{bzp} and \cite{kellyast} easily yield that $\mgr(X)$ belongs to the subcategory of $\dmger$ densely generated by motives of projective varieties of dimension at most $\dim X$ indeed. 

\ref{imotw2} 
 The twist functor $-\lan 1\ra$ obviously sends the subcategory $\chower$ into itself. 
  Applying Proposition \ref{ploc}(\ref{iwl31}) we obtain that $\wchowsg$ restricts to the subcategory $\dmger\lan 1\ra$, and  there exists a weight structure $\wgbirp$  on $ \dmgbirp=\dmger/\dmger\lan 1 \ra$ such that the localization functor is weight-exact.
		
\ref{imotw3}. According to Proposition \ref{pbw}(\ref{ineg},\ref{igen}), 
 $\wgbirp$ extends to a weight structure $\wgbir$ on  $\dmgbir$  indeed. On the other hand,  Theorem 5.2.1 of \cite{bososn} gives a bounded weight structure $w$ on $\dmgbir$ whose heart 
	 equals $\chowbir$. Since $\pi(\chower)\subset \chowbir$, 
	 Proposition \ref{pbw}(\ref{igen})  yields that $w=\wgbir$.
	
	\ref{imotp2}. 
Since $\mgb(X)$ is an object of $\dmgbirp$, the previous assertion  implies that  $\mgb(X)\in \dmgbirp{}_{\wgbirp=0}$. Moreover, $\mgr(X)\in \dmger{}_{\wchowsg\le 0}$ by assertion \ref{imotp1}.    Thus we can apply Proposition \ref{ploc} (for $\cu=\dmger$ and $\cp=\obj \chower\lan 1\ra$) to obtain the existence of $M$ and $h$  as desired. 
	\end{proof}

\begin{rema}\label{rcompact}
 The proof of  the aforementioned Proposition 5.3.3 of \cite{kellyast} and  Theorem 2.2.1(1) of  \cite{bzp} relied on Gabber's resolution of singularities (see Theorem X.2.1 of \cite{illgabb}). Thus if some "$p$-adic version" of loc. cit. was available (in the case $p>0$) then one would be 
 able to extend all the statements of this section to the case of an arbitrary coefficient ring $R$. 

On the other hand, if $p=0$ then Hironaka's resolution of singularities gives a smooth compactification $P$ for (every) $X\in\sv$, and the Gysin distinguished triangle in Proposition \ref{prmot}(\ref{imot5}) implies that $\co(\mgr(X)\to \mgr(P))\in \obj \dmger\lan 1 \ra$. This makes our results much less interesting in the case $p=0$.
\end{rema}


Now we will relate motivic categories to unramified cohomology. 
 We will use several results of 
  \cite{kabir} where only the case $R=\z$ was considered; note however that the proofs generalize to the case of an arbitrary coefficient ring $R$ without any problems. 

\begin{theo}\label{tbir}
Let $N$ be an object of $\dmer$,  $S$ is  a sheaf with transfers (see Proposition \ref{prmot}(\ref{imot6})),  and $X\in \sv$.

\begin{enumerate}
\item\label{ibir1} There exists an exact fully faithful functor $\ii$ that is right adjoint to $\pi$. 

Moreover, 
 $N$ belongs to the essential image of $\ii$ if and only if the homomorphism $\dmer(\mgr(f),N[i])$ 
is bijective for any $i\in \z$ and any open dense embedding $f$ of smooth $k$-varieties; if this is the case then we will say that $N$ is {\it birational}. 

\item\label{ibir2} 
 $S$ is birational if and only if the homomorphism $\dmer(\mgr(f),S)\cong \for(S)(f)$ (see Proposition \ref{prmot}(\ref{imot6}) once again)  is bijective for any $f$ as above. Moreover, if this is the case then $\dmer(\mgr(X),S[i]) = 
\ns$ for any $i\neq 0$. 

\item\label{ibir3} 
 The subcategory $\hiz$ of  birational sheaves is a Serre  subcategory of $\hir$, and the embedding $\hiz\to \hir$ possesses a right adjoint functor $\rnrz$ that sends a sheaf with transfers into its maximal birational subsheaf. 

Moreover,  for the 
 counit homomorphism $c(S):\rnrz(S)\to S$ the image of  $\for(c(S))(X)$ gives the {\it unramified 
 part of $\for(S)(X)$} 
 in the sense of \cite[Definition 7.2.1]{kabir}.

\item\label{ibir4} 
Assume that 
$M\in \obj \chower$ 
  and for $h\in \dmger(\mgr(X),M)$ the morphism $\pi(h)$ is an isomorphism.
Then the homomorphism  
$\dmer(h,S)$ is injective, and its image coincides with the  aforementioned unramified 
 subgroup of  $\dmer(\mgr(X),S)\cong \for(S)(X)$. 

\item\label{ibirsh} The restriction of the  functor $\dmer(-,\rnrz(S))$ to $\dmger$ is isomorphic to $H_{S_{\chower}}$ (see Proposition \ref{pwt}(\ref{iwcpu})), where $S_{\chower}$ is the restriction of $\dmer(-,S)$ to $\chower$.

\item\label{ibirc}  A cohomological functor from $\dmger$ into an abelian category $\au$ is birational (i.e., $H(\mgr(f)[-i])$ is an isomorphism any open dense embedding $f$ of smooth $k$-varieties and all $i\in \z$) if and only if 
 $H$ annihilates $\dmger\lan 1\ra$; these conditions are fulfilled  if and only if $H(\mgr(P)\lan 1 \ra[-i])=0$ for all $P\in \spv$ and $i\in \z$.

In particular, $N$ is birational (in the sense of assertion \ref{ibir1}) if and only if $\mgr(P)\lan 1\ra\perp N[i]$ for any $i\in \z$ and $P\in \spv$, and for an additive functor $F:\chower{}\opp\to \au$ the corresponding functor $H_F$ (see Proposition \ref{pwt}(\ref{iwcpu})) is birational if and only $F$ kills all  $\mgr(P)\lan 1\ra$.
\end{enumerate}
\end{theo}
\begin{proof}
\ref{ibir1}.  The existence of $\ii$ is an immediate consequence of well-known abstract nonsense; see Theorem 4.3.5 of \cite{kabir}. 
The second part of the assertion is obvious. 

\ref{ibir2}. Easy from Theorem 4.2.2(e) of ibid. (along with Proposition \ref{prmot}(\ref{imot5})).

\ref{ibir3}. $\hiz$ a Serre  subcategory of $\hir$ according to Proposition 2.6.2 of 
 ibid. The right  adjoint functor $\rnrz$ exists according to Proposition 2.6.3 of ibid., and it gives the maximal birational subsheaf since $\hiz$ is a Serre subcategory.

Lastly, 
$c(S)$ gives the corresponding unramified 
 subgroup according to Theorem 7.3.1 of ibid.

\ref{ibir4}. Assume that $M$ is a retract of $\mgr(P_M)$ for $P_M\in \spv$. Then  
the homomorphism $\dmger(\mgr(P_M),c(S))$ 
 is bijective by Corollary 7.3.2 of ibid. Thus 
 we have a  commutative square
$$\begin{CD}
 \dmer(M,\rnrz(S)) @>{\cong}>>\dmer(\mgr(X),\rnrz(S)) \\
@VV{\cong}V@VV{}V \\
\dmer(M,S)@>{}>>\dmer(\mgr(X),S)
\end{CD}$$
that clearly gives the result in question.

\ref{ibirsh}. To apply Proposition \ref{pwt}(\ref{iwcpu}) we should verify that the restriction of $\dmer(-,\rnrz(S))$ 
  to $\chower[i]\subset \dmer$ is zero if $i\neq 0$ and is isomorphic to $H_{S_{\chower}}$ for $i=0$.

The first of this statements immediately follows from assertion \ref{ibir2}, and the second one 
 is an easy consequence of assertion \ref{ibir4} (as well as of  \cite[Corollary 7.3.2]{kabir}).

\ref{ibirc}. If $f$ is a smooth dense embedding then Proposition \ref{prmot}(\ref{imot5}) easily implies that $\co(\mgr(f))\in \obj \dmger\lan 1\ra$. On the other hand, if $X\in \sv$ then $f_X:\afo\times X\to \p^1\times X$ is a dense embedding, and $\co(\mgr(f_X))\cong \mgr(X)\lan 1\ra$. The first equivalence in the assertion follows easily.

Conversely, if $f$ is a smooth dense embedding then Proposition \ref{prmot}(\ref{imot5}) easily implies that $\co(\mgr(f))\in \obj \dmger\lan 1\ra$. Theorem \ref{tmotw}(\ref{imotp1}) implies that the set of all $\mgr(P)\lan 1\ra$ densely generates the category $\dmger\lan 1\ra$; this gives the converse implication.

The "in particular" statements are easy as well; we only note that the functor $H_F$ annihilates $\chower[i]$ for all $i\neq 0$ (see Proposition \ref{pwt}(\ref{iwcpu})). 
 \end{proof}

Now let us apply Theorem \ref{tbir} to the study of Brauer groups. We will also relate it to Theorem 5.1 of \cite{eked} later.

\begin{coro}\label{cbrauer}
Assume that $R$ is a localization of $\zop$, 
 and set $P^{\bra} = \rnrz(S^{\bra})$ (see Proposition \ref{prmot}(\ref{imotet}) for the definition of $S^{\bra}$ along with $\rgm$).

1. Then $P^{\bra}$ is a birational sheaf with transfers (consequently, it annihilates $\chower\lan 1\ra$), and $\for(P^{\bra})$ is the {\it unramified Brauer group sheaf}, that is, it sends a connected smooth $k$-variety $Y$ into the unramified part of the group $\bra(\spe(k(Y)))$. 

2.  The restriction $H^{\bra}_{nr}$ of the functor  $ \dmer(-,P^{\bra})$ to $\dmger$ is isomorphic to the functor
$H_{S^{\bra}{\chower}}\cong H_{\rgm[2]_{\chower}}$, where $S^{\bra}_{\chower}$ (resp. $\rgm[2]_{\chower}$)  
 is the restriction to $\chower$ of the functor $\dmer(-,S^{\bra}) $ (resp. $\dmer(-,\rgm[2]) $).


Moreover, if $N=\mgr(X)$ for some $X\in \sv$ and $h:N\to M\in \dmger{}_{\wchowsg=0}$ is the morphism provided by Theorem \ref{tmotw}(\ref{imotp2}) then  $H^{\bra}_{nr}(N)\cong \imm (\dmger(h,S^{\bra}))\cong \imm  (\dmger(h,\rgm[2]))$. 

 \end{coro}
\begin{proof}
All the statements easily follow from   Theorem \ref{tbir}  combined with Proposition \ref{prmot}(\ref{imotet}). 
\end{proof}

\subsection{On the relation to motivic Euler characteristics of varieties}\label{smgc}

Starting from this moment we will assume that $R$ is a $\zop$-algebra; the main case is just $R=\zop$.

Let us recall some properties of motives with compact support.


\begin{pr}\label{pmgc}
Denote by   $\mgcr$ the $R$-linear  motif with compact support functor from the category $\schpr$ of $k$-varieties with morphisms being proper morphisms of varieties into $\dmer$; this functor  is  essentially provided by  \S4.1 of \cite{1} along with \S5.3 of \cite{kellyast} (see the proof below). 

Then $\mgcr$ satisfies the following properties.

\begin{enumerate}
\item\label{imceq}
We have $\mgcr(P)=\mgr(P)$  whenever $P\in \spv$.\footnote{More generally, $\mgcr(X)\in \obj \dmger$ for any $X\in \var$. However, we will not consider motives of singular varieties in this paper.}

\item\label{imctr} If $i:Z\to X$ is a closed embedding of $k$-varieties and $U=X\setminus Z$ then there exists a distinguished triangle 
\begin{equation}\label{eimctr}
\mgcr(Z)\stackrel{\mgcr(i)}{\longrightarrow} \mgcr(X)\to \mgcr(U)\to \mgcr(Z)[1].\end{equation}

\item\label{imcpr} If $X,Y\in \var$ then $\mgcr(X\times Y)\cong \mgcr(X)\otimes \mgcr(Y)$.


\item\label{imceuler} Consequently, the function sending $X\in \var$ into the class of $\mgcr(X)$ in $K_0^{\tr}(\dmger)$  sends products of varieties into products in $K_0^{\tr}(\dmger)$ (see Remark \ref{reuler}(3)) and satisfies the {\it scissors} relation $[X]=[U]+[Z]$ under the assumptions of assertion \ref{imctr}.

\item\label{imcp} Thus 
 for any additive functor $F:\chower\to \au$, where $\au$ is an additive category, the correspondence $X\mapsto F_{K_0}([\mgcr(X)])$ (see Theorem \ref{teuler}(1)) satisfies the scissors relation  as well. 

\end{enumerate}

\end{pr}
\begin{proof}
In Definition 5.3.1, Lemma 5.3.6, Proposition 5.3.12(1) (combined with Theorems 5.2.20, 5.2.21, and 5.3.14), 
Proposition 5.3.5,  Proposition 5.3.8, and Corollary 5.3.9 of \cite{kellyast}, respectively, the obvious $\zop$-linear analogues of 
 assertions \ref{imceq}--\ref{imcpr}  
 were justified. 

The $R$-linear results in questions can either be proved similarly or deduced from the results of ibid. using the properties of the obvious connecting functor $-\otimes_{\zop}R:\dmeop\to \dmer$  given by Proposition 1.3.3 of \cite{bokum}.

Lastly, assertions \ref{imceuler} and \ref{imcp} follow from assertions \ref{imctr} and \ref{imcpr} immediately.
\end{proof}

\begin{rema}\label{rekest}
If $k$ is characteristic $0$ field then Hironaka's resolution of singularities easily implies that any function from $\sv$ into an abelian group $A$ that satisfies the scissors relation is uniquely determined by its values on smooth projective varieties.

Using this observation one can obtain an extension of 
  \cite[Theorem 3.4(i)]{ekest} to fields of positive characteristics.
	
	We will discuss other functions of this sort in Theorem \ref{tceked} and Remark \ref{reked}; see also Proposition 4.2 of \cite{scav}.
\end{rema}


In order to relate motives with compact support with earlier results we recall some basics on the twist-stable motivic category
$\dmgr$.  Proposition \ref{prmot}(\ref{imot4}) implies that the obvious functor from $\dmer$ into the  $2$-colimit category 
$\inli (\dmger\stackrel{\lan 1 \ra}{\longrightarrow}\dmger\stackrel{\lan 1 \ra}{\longrightarrow}\dmger\stackrel{\lan 1 \ra}{\longrightarrow}\dots$ 
is a full exact embedding (of triangulated categories); cf. \S2.1 of \cite{1} or \S6.1 of \cite{bev}. Consequently, we will assume that $\dmger$ is  a full strict subcategory of $\dmgr$. Moreover, $\dmgr$ a monoidal category (see loc. cit.), this embedding is monoidal, and for any $n\ge 0$ the functor $-\lan n\ra_{\dmgr}=-\otimes R\lan n \ra$ has an obvious inverse that we will denote by  $-\lan -n\ra$. 


\begin{pr}\label{pstab}
Assume that $X,Y\in \sv$, and $\dim X\le n$ (for some $n\ge 0$). 
\begin{enumerate}
\item\label{istabr}
$\dmgr$ is a rigid category, i.e., all its objects are dualizable.

\item\label{istabd} If 
 the connected components of $X$ are of dimension $n$ then $\mgcr(X)\cong D({\mgr(X)})\lan n \ra$, where $D=D_{\dmgr}$ is the duality functor.

\item\label{istab1} There exists a  unique bounded weight structure $\wchowg$ on $\dmgr$ whose heart is the category 
$\chowr=\chower[\lan -1\ra]$ (i.e., this is the full strict subcategory of $\chower$ whose object class equals $\cup_{n\ge 0} \obj \chower\lan -n \ra $).

\item\label{istab2}
The auto-equivalences  $-\lan n \ra$ are weight-exact with respect to  $\wchow$ for all $n\in \z$.

\item\label{istab3}
The duality functor $D_{\dmgr}$ is weight-exact (as well), i.e., it sends $\dmgr{}_{\wchowg\le 0}$ into $\dmgr{}_{\wchowg\ge 0}$ and vice versa (cf. Proposition \ref{pbw}(\ref{idual})).

\item\label{istabdu} Duality restricts to the subcategory $\chowr$ of $\dmgr$, and for the weight complex functor $\tst:\dmgr\to K^b(\chowr)$ we have $D_{K^b(\chowr)}\circ \tst\cong \tst\circ D_{\dmgr}$.

\item\label{istab4} 
 $K_0^{\tr}(\dmgr)\cong K_0(\dmger)[R\lan 1\ra]^{-1}\cong K_0^{\add}(\chowr) \cong K_0^{\add}(\chower)[R\lan 1\ra]^{-1}$.




\item\label{istabkz} Denote the category of finitely generated $\zl$-modules by $\zlmod$. Then $K_0^{\add}(\zlmod)$ is the free abelian group generated by the classes of $\zl$ and of $\zl/\ell^i\zl$ for all $i>0$.

\item\label{istabet} Assume that $k$ is an algebraically closed field, 
  $R$ is a localization of $\zop$, and a prime $\ell$ belongs to $ \p\setminus\{p\}$ and is not invertible in $R$.

Then there exists an \'etale realization functor $RH_{et,\zl}:\dmgr\to D^b(\zl)$ that possesses the following properties:

$RH_{et,\zl}(\mg(X)) \cong RH(X_{et},\zl)$, $RH_{et,\zl}\circ- \lan n \ra\cong RH_{et,\zl}\circ [2n]$,\footnote{Certainly, these isomorphisms are not canonical, but this does not make any difference for our purposes. Similarly, in (\ref{egm}) we essentially choose a non-canonical identification  $RH_{et,\znz}\circ- \lan 1 \ra\cong RH_{et,\znz}\circ[-2]$.}  and $RH_{et,\zl}\circ D_{\dmgr}\cong D_{D^b(\zl)}\circ RH_{et,\zl}$.

Moreover, if for all $i\in \z$ the functor $H^i_{et,\zl}$  is the  $i$-th cohomology of $RH_{et,\zl}$ applied to objects of $\dmger$ then there exist 
 $\znz$-\'etale realization functors $H^i_{et,\znz}:\dmger\to \znzmod\subset\zlmod$ along with the following long exact sequences of functors from $\dmger$: 

$\dots\to  H^i_{et,\zl}\stackrel{\times \ell^n}{\longrightarrow}H^i_{et,\zl}\to H^i_{et,\znz}\to  H^{i+1}_{et,\zl}\stackrel{\times \ell^n}{\longrightarrow}H^{i+1}_{et,\zl}\to\dots$ 
 and  \begin{equation}\label{egm} \dots\to  H'{}^{i-1}\stackrel{\times \ell^n}{\longrightarrow}H'{}^{i-1}\to H^i_{et,\znz}\to  H'{}^{i}\stackrel{\times \ell^n}{\longrightarrow}H'{}^{i}\to\dots,\end{equation}  where $H'{}^j$ for $j\in \z$ denotes the restriction of the functor $\dmer(-,\rgm[j])$ 
 to $\dmger$, and $\rgm$ is the object of $\dmer$ 
 mentioned in Proposition \ref{prmot}(\ref{imotet}). 
\end{enumerate}
\end{pr}
\begin{proof}  
\ref{istabr}, \ref{istabd}. These statements easily follow from their $\zop$-linear versions provided by Theorem  5.3.18 of \cite{kellyast}. 

\ref{istab1}. 
Proposition \ref{prmot}(\ref{imot7}) easily implies that the subcategory $\chowr$ of $\dmgr$ is negative, whereas  Theorem \ref{tmotw}
(\ref{imotp1}) yields that $\chowr$ strongly generates $\dmgr$. Thus Proposition \ref{pbw}(\ref{igen}) gives the result.

\ref{istab2}, \ref{istab3}. 
 By Proposition \ref{pbw}(\ref{iwex},\ref{idual}), it suffices to verify that twists and duality send Chow motives into Chow ones.  This property is obvious for twists, and for duality it  is immediate from assertion \ref{istabd} along with Proposition \ref{pmgc}(\ref{imceq}).

Moreover, we have just checked the first statement in assertion \ref{istabdu}. The second part of the assertion now follows from Proposition \ref{pwt}(\ref{iwcfunct}).

\ref{istab4}. Immediate from our definitions along with Theorem \ref{teuler}(4).


\ref{istabkz}. Obvious; cf. Proposition 3.3 of \cite{ekest}. 

\ref{istabet}. These statements can be deduced from the results  of \cite{kellyast} if one argues similarly to (\S3.3 of) \cite{1}; they also easily follow from the (much more general) Theorem 7.2.11 of \cite{cdet} as well as from Remark 4.8 of \cite{ivorra}. 
\end{proof}

Now let us apply this statement to the calculation of certain $K_0$-classes. 

\begin{theo}\label{tceked}
Adopt the notation and the assumptions of Proposition \ref{pstab}(\ref{istabet}) (in particular, $R$ is a localization of $\zop$, and one may assume that $R=\zll$ of $R=\zop$); assume that all the connected components of $X\in \sv$ have dimension $d$, and let $N$ be an object of $\dmger$. 
 
For $n\ge 0$\footnote{One can also take $n$ to be negative here; yet all the classes mentioned in this proposition vanish in this case.} take the functors $E^n$, $H^n$, $F^n$, and $G^n$ to be the restrictions to $\chower$  of the functors $H^{2d-n}_{et,\zl}$, $H^n_{et,\zl}$, 
$\torr H^{n}$, and $H^{n}/\torr H^{n}$, respectively; here we assume that the target of these functors equals $\zl-\modd\opp$.

Then the following statements are valid.

\begin{enumerate}
\item\label{iek1}   The class $E^n_{K_0}([\mgcr(X)])\in K_0^{\add}(\zlmod\opp)=K_0^{\add}(\zlmod)$
 equals  $F^{n+1}_{K_0}([\mgr(X)])+G^{n}_{K_0}([\mgr(X)])$. 

\item\label{iek2}  
 $G^{n}_{K_0}([N])=m[\zl]$, where $m\in \z$ satisfies the equality $G^{n}\otimes \q_{K_0}([N])=m[\ql]\in K_0(\ql-\operatorname{mod})$. Moreover, 
we have $G^{n}\otimes \q_{K_0}([N])= \sum_{i=-\infty}^{\infty}(-1)^{i} [H^i_{G^{n}\otimes \q}(N)]$; here  $H_{G^{n}\otimes \q}$    
    is the cohomological functor from $\dmger$ into $\ql$-vector spaces that corresponds to $G_n\otimes \q$ according to Proposition \ref{pwt}(\ref{iwcpu}). 

\item\label{iek2c} Let $N=\mgr(X)$. 

 Then $H^i_{G^{0}\otimes \q}(N)=\ns$ if $i>0$, 
  and $H^0_{G^{0}\otimes \q}(N)\cong \ql^c$, where $c$ is the number of connected components of  $X$. 

Furthermore, if $k$ is of finite transcendence degree over its prime subfield 
   then $H^i_{G^{n}\otimes \q}(N)\cong \grwd_n(H^{i+n}_{et}(X,\ql))$; that is, we take the factors of the Deligne's weight filtration (for this purpose we take a field of definition $k'\subset k$ of $X$ such that $k'$ is finitely generated over its prime subfield).  Consequently, $m= \sum_{j=0}^{2d}(-1)^{j}\dim_{\ql}\grwd_n(H^j_{et}(X,\ql))$.


\item\label{iek3} 
  $F^{0}=F^{1}=0$, whereas 
	 the functors   $F^{2}$, $F^3$, $G^0$, and $G^1$ kill $\chower\lan 1\ra$. 

	Consequently, the functor $H_A$ is birational (in the sense of Theorem \ref{tbir}(\ref{ibirc})) and   
	 $A_{K_0}([\mgr(X)])=[H_{A}(\mgr(X))]=[A(M)]$ if $A$ is  one of the functors   $F^{2}$, $F^3$, $G^0$, or $G^1$, 
	and $h:N\to M\in 
	\obj \chower$ is the morphism provided by Theorem \ref{tmotw}(\ref{imotp2}). 
	
\item\label{iekbg} 
Assume that $N=\mgr(X)$.  If there exists a finite flat morphism $Y\to X$, where $Y=\af^d\setminus Z$, then $H^i_{G^{n}\otimes \q}(N)=\ns$ for any $i\in \z$ if $0<n<2\codim_{\af^d}(Z)$, and $H^i_{G^{0}\otimes \q}(N)=\ns$ for $i\neq 0$ as well. 

More generally, if
$X$ is unirational 
 then $H_{G^{1}\otimes \q}(N)=\ns$, and also $H_{F^3}(N)\cong 
H^{\bra}_{nr}(N)\otimes \zll\cong \dmer(N,\rgm[2])\otimes \zll$ (see Corollary \ref{cbrauer}); thus  $H_{F^3}(N)$ is the $\ell$-adic part of the unramified Brauer group of $k(X)$).  
\end{enumerate}
\end{theo}
\begin{proof}
\ref{iek1}. 
Applying  Proposition \ref{pwt}(\ref{iwcfunct}) one easily obtains that the assertion  
 is an easy application of duality provided by 
  Proposition \ref{pstab}(\ref{istabd},\ref{istab3},\ref{istabet}) along with properties of twists mentioned in parts \ref{istab2} and \ref{istabet} of the proposition.

\ref{iek2}. Obviously, there does exist $m\in \z$ such that $G^{n}_{K_0}([N])=m[\zl]$, and we also have $G^{n}\otimes \q_{K_0}([\mgr(X)])=m[\ql]$. 

Next,  $G^{n}\otimes \q_{K_0}([N])= \sum_{i=-\infty}^{\infty}(-1)^{i} [H^i_{G^{n}\otimes \q}(N)]$ by Theorem \ref{teuler}(3). 

\ref{iek2c}. $H^i_{G^{n}\otimes \q}(N)=0$ for $i<0$ since $t^{st}(N)$ is homotopy equivalent to a complex concentrated in non-negative degrees; see Proposition \ref{pwt}(\ref{iwcons}) along with Theorem \ref{tmotw}(\ref{imotp1}). 

Next,  Proposition \ref{pwt}(\ref{iwcpu}) implies that $H^0_{G^{0}\otimes \q}$ is the only cohomological functor from $\dmger$ that restricts to $0$ on $\chower[i]$ for $i\neq 0$ and whose restriction to $\chower$ gives the functor $H^0\otimes \q$. On the other hand, tensoring the restriction of the zeroth  cohomology of $RH_{et,\zl}$ (see Proposition \ref{pstab}(\ref{istabet})) to $\dmger$ by $\q$ one obtains a functor that possesses these properties as well. Hence $H^0_{G^{0}\otimes \q}(N)\cong \ql^{c}$ indeed.

The remaining calculation is an easy application of the theory of weight spectral sequences that was introduced in \cite{bws}.  Recall that  for any cohomological functor $H$ from $\dmger$ into an abelian category $\au$ and any $N\in \obj \dmger$ 
	 there exists a spectral sequence $T=T_w(H,N)$ with $E_1^{pq}=H^{q}(N^{-p})$,  such that $N^i$ and the boundary morphisms of $E_1(T)$ come from  $t^{st}(N)$; it converges to $H^{p+q}(N)$ (see Proposition 1.4.1(1) of \cite{bwcp} or Theorem 2.4.2 of \cite{bws}; note that we have convergence for any $H$ since $\wchowsg$ is bounded). 
Now, easy weight arguments described in Remark 2.4.3 of ibid. and (especially) in Remark 3.5.2(3) of \cite{bscwh} yield that this spectral sequence degenerates at $E_2$ if $H=H^0_{et,\ql}$. Moreover, the corresponding filtration on $H^{p+q}(N_i)$ is the Deligne's weight one up to a shift (that is described in loc. cit.) if one takes $N_i=\mgr(X)[-i]$ for $i\in \z$,  computes $H$ via the functor $H^0_{et,\ql,k'}:\dmger(k')\to \q_l[G]-\modd,\ N'\mapsto H^0_{et}(N'\otimes_{\spe k'}\spe k)$, where $G=\operatorname{Gal}(k')$,  and uses the isomorphism $H^0_{et,\ql,k'}(M^R_{gm,k'}(X_{k'})[-i])\cong H^i(X_{et},\ql)$; cf. Proposition 4.1.6(1) of ibid. Combining these observations we obtain the result easily.

	\ref{iek3}. 
	Obviously, $F^{0}=0$. Next, the vanishing of $F^{1}$ can be easily checked using Poincare duality; cf. the proof of assertion \ref{iek1}. 
	
	Hence applying Proposition \ref{pstab}(\ref{istabet}) we obtain that  the functors $G^0$, $G^1$, $F^{2}$, and $F^3$ kill $\chower\lan 1\ra$.
	
	 Lastly, Theorem \ref{tbir}(\ref{ibirc}) implies that $H_{A}$ is birational for any $A$ as above. Thus applying 
	Theorem  \ref{teuler}(2)  
	 we obtain that $A_{K_0}([\mgr(X)])=[H_{A}(\mgr(X))]=[A(M)]$ 
	 indeed.
	
	\ref{iekbg}. First we study the cohomology of  $N'=\mgr(Y)$ (for $Y=\af^d\setminus Z$).
	
	It is easily seen that we can assume  the field $k$ to be of finite transcendence degree over its prime subfield  when computing   $H^i_{G^{n}\otimes \q}(N')$; cf. the proof of   \cite[Theorem 2.5.4(II.1)]{bmm} for a closely related argument. Thus one can use the Gysin exact sequence for \'etale cohomology to check  that if $n<2\codim Z$ then   $\grwd_n(H^{j}_{et}(Y,\ql))=\ns$ unless $(n,j)=(0,0)$. Hence applying assertion \ref{iek2c} we obtain the corresponding vanishing of  $H^i_{G^{n}\otimes \q}(N')$.\footnote{Alternatively, one may obtain this fact (along with the similar vanishing of $H^i_{G^{n}}(N')$) more directly by using the Gysin distinguished triangle for motives 
	along with the fact that effective Chow motives have no \'etale cohomology of negative degrees.}
	
	Next recall that the morphism $\deg f\cdot\id_N$ factors through $N'$ whenever  $f$ is a finite flat morphism $Y\to X$; see Proposition \ref{prmot}(\ref{imotr}). Thus if $n<2\codim Z$ then $H^i_{G^{n}\otimes \q}(N)=\ns$ unless  $(n,i)=(0,0)$ indeed.
	
	Now 
	 let $X$ be an arbitrary (smooth) unirational $k$-variety. 
	
	Corollary \ref{cbrauer}(2) implies that the functor $H^{\bra}_{nr}\otimes \zll$ is birational, 
	 and the functors 	$H_{G^{1}\otimes \q}$ and $H_{F^3}$ also are according to the previous assertion. Hence 
	 it suffices to consider the case where there exists a finite flat morphism $Y\to X$ such that $Y$ is an open (dense) subvariety of $\af^d$. As we have just proved, this implies $H^i_{G^{1}\otimes \q}(N)=\ns$ for all $i\in\z$. Next, $H^{\bra}_{nr}(N')\cong H^{\bra}_{nr}(\mgr(\af^d))\cong H^{\bra}_{nr}(R)=\ns$; thus $H^{\bra}_{nr}(N)$ is of finite exponent by the same factorization argument as above.
	
	Recall also that for $h:N\to M$ as in the previous assertion we have $H_{F^3}(N)\cong H_{F^3}(M)$ and  $H^{\bra}_{nr}(N)\cong H^{\bra}_{nr}(M)$ (since $\co(h)\in \obj\dmger\lan 1\ra$). Thus  $H^{\bra}_{nr}(M)\cong \dmer(M,S^{\bra})$ is of finite exponent as well. Since $M$ is a retract of $\mgr(P_M)$ for some smooth projective $P_M$, 
	 Proposition \ref{prmot}(\ref{imotet}) implies that 
	 $H^{\bra}_{nr}(M)$ is also isomorphic to  $\dmer(M,\rgm[2])$ in the notation on this proposition.  
	
	Now we note that by passing to the direct limit with respect to $n$ in the long exact sequences given by this proposition
	one can  obtains an exact sequence 
	\begin{equation}\label{ecofin} 
	\ns \to (\q_l/\zl)^{c_M}\to \dmer(M,\rgm[2])\otimes \zll \to F^3(M)\to \ns; \end{equation}  
	this is the summand corresponding to $M$ of the exact sequence (4.3) of \cite{collbr} for the variety $P_M$. Since the group $\dmer(M,\rgm[2])$ is of finite exponent, $c_M=0$, and we obtain $\dmer(M,\rgm[2])\otimes \zll\cong F^3(M)$. Thus $H_{F^3}(N)\cong \dmer(M,\rgm[2])\otimes \zll\cong H^{\bra}_{nr}(N)\otimes \zll$ indeed. 
\end{proof}

\begin{rema}\label{reked}
\begin{enumerate}
\item\label{iur} Putting assertions \ref{iek1} and \ref{iekbg} together we obtain that for any unirational variety $X$ we have $E^2_{K_0}([\mgcr(X)])=[H^{\bra}_{nr}(\mgr(X))\otimes \zll]$. Moreover,  for all $n>0$ the classes $E^n_{K_0}([\mgcr(X)])$ are torsion, that is, they belong to the subgroup of $K_0^{\add}(\zlmod)$ generated by $[\zl/\ell^i\zl]$ for all $i>0$; see Proposition \ref{pstab}(\ref{istabkz}).

\item\label{icofin} 
Furthermore, our arguments demonstrate that $E^2_{K_0}([\mgcr(X)])$ is closely related to $H^{\bra}_{nr}(\mgr(X))\otimes \zll$ for a general smooth $X$ as well; see (\ref{ecofin}). 
 Moreover, 
 it can make sense to consider the class of $H^{\bra}_{nr}(\mgr(X))\otimes \zll$ in $K_0^{\add}(\au)$, where $\au$ is the category of $\zl$-modules of {\it cofinite type}, that is, of modules of the form $T\bigoplus (\ql/\zl)^c$, where $c\ge 0$ and $T$ is finite. 

 \item\label{irek1} Now recall that in 
 \cite{ekest} and \cite{eked} 
 certain Grothendieck group of stacks was studied by means of Euler characteristic homomorphisms. 

Moreover, if $G$ is a finite group then an images the class $[BG]$ with respect to a homomorphism of this sort coincides with that of $[X]$, where $X=Y/G$, and $Y= \af^d\setminus Z$ is an open subvariety such that $Z$ is of codimension large enough, $G$ acts linearly on $\af^d$, and this action restricts to a free action on $Y$.

As demonstrated in Proposition 4.2 of \cite{scav} (cf. Remark \ref{rekest}), this Euler characteristic of $BG$ can be computed as $H_{K_0}(\mgcr(X))$, where $H$ is a certain functor from $\chower$. In particular, in \S5 of \cite{eked} $k$ was assumed to be equal to $\com$,  and the corresponding functor $H$ is the singular cohomology $H_{sing}^{2d-n}$ 
 (with values  in finitely generated abelian groups) one.

Thus our calculations above (see also part \ref{iur} of this remark) enable one 
  to extend most of  \cite[Theorem 5.1]{eked} to fields of arbitrary characteristics. Note here that if $X=Y/G$ as above, and $G$ is a (finite) group  
	 of order prime to $\cha k=p$, the unramified Brauer group of $X$ has a rather easy description solely in terms of $G$; see Theorem 12 of \cite{salt} that generalizes the corresponding Bogomolov's calculation (made in the case $k=\com$).  Respectively, it is known in several cases that   $E^2_{K_0}([\mgcr(X)])\neq 0$ for  $X$ of this sort. 

 \item\label{ire1pic} In contrast to Theorem 5.1 of \cite{eked}, we cannot say much on the class  $E^1_{K_0}([\mgcr(X)])=F^2_{K_0}([\mgr(X)])$ (for $X=Y/G$ as above). This appears to be related to the fact that $H^2_{et}(X',\zl)$ does not have to be torsion-free for a smooth projective unirational $X'$ over a 
an algebraically closed field of characteristic $p\neq 0,\ell$. 

\item\label{irekrck} The author wonders whether it suffices to assume that $X$ is {\it rationally (chain) connected} in part \ref{iekbg} of our theorem. Note that if $p=0$ then $X$ possesses a smooth compactification $X'$ (by Hironaka's resolution of singularities); then $X'$ is rationally (chain) connected as well. Next, and one can take $M=\mgr(X')$ in our argument (see Remark \ref{rcompact}), and $H^{\bra}_{nr}(M)$ is torsion (see Lemma 2.6 of \cite{kahntorsurf}). 
However, the author does not know how to extend this observation to the case $p>0$. 

 \item\label{irekgs} In the case $p=0$  and $X\in \sv$ the classes $[\mgcr(X)]\in K_0^{\add}(\chowez)$ and $F_{K_0}([\mgcr(X)])$ for certain additive functors from $\chowez$ 
 were   essentially studied in \S\S3.2--3.3 of \cite{gs}; see Proposition 6.6.2 of \cite{mymot} for a justification of this claim.
\end{enumerate}
\end{rema}

\section{Supplements}\label{ssupl}

In \S\ref{sbr} we 
 prove Proposition \ref{prmot}(\ref{imotet}) and discuss 
 some ideas for extending our results to the case where $p$ is (positive and) not invertible in $R$. For this purpose we also give some (more) properties of weight-exact localizations.

In \S\ref{skz} we describe a certain generalization of the results of \S\ref{stable} to the case where no $\infty$-model is available for $\cu$ (and $w$ is not necessarily bounded). In particular, we calculate the Grothendieck group of an (arbitrary) bounded weighted category; this statement generalizes Theorem 5.3.1 of \cite{bws}.

\subsection{More on the Brauer group sheaf and 
 motives}\label{sbr}

We start from recalling some basics on $t$-structures. Here we use the so-called homological convention for them; note that it is "opposite" to the cohomological convention that was applied in \cite{bbd} and in several previous papers of the author.

\begin{defi}\label{dtstrh}

1. A couple of subclasses  $\cu_{t\le 0},\cu_{t\ge 0}\subset\obj \cu$ will be said to be a
$t$-structure $t$ on $\cu$  if 
they satisfy the following conditions:

(i) $\cu_{t\le 0}$ and $\cu_{t\ge 0}$  are strict, i.e., contain all
objects of $\cu$ isomorphic to their elements.

(ii) $\cu_{t\le 0}\subset \cu_{t\le 0}[1]$ and $\cu_{t\ge 0}[1]\subset \cu_{t\ge 0}$.

(iii)  $\cu_{t\ge 0}[1]\perp \cu_{t\le 0}$.

(iv) For any $M\in\obj \cu$ there exists a  {\it $t$-decomposition} distinguished triangle
\begin{equation}\label{tdec}
L_tM\to M\to R_tM{\to} L_tM[1]
\end{equation} such that $L_tM\in \cu_{t\ge 0}$ and $ R_tM\in \cu_{t\le 0}[-1]$.

2. $\hrt$ is the full subcategory of $\cu$ whose object class is $\cu_{t=0}=\cu_{t\le 0}\cap \cu_{t\ge 0}$.

3. We will say that a class $\cp\subset \obj \cu$ {\it generates $t$} if $\cu_{t\le 0}=(\cup_{i>0}\cp[i])\perpp$. 
\end{defi}

Let us recall some properties of $t$-structures.

\begin{pr}\label{ptstr}
Let $t$ be a $t$-structure on $\cu$. Then the following statements are valid.

\begin{enumerate}
\item\label{itcan}
The triangle (\ref{tdec}) is canonically and functorially determined by $M$. Moreover, the functor $L_t$ is right adjoint to the embedding $ \cu_{t\ge 0}\to \cu$ (if we consider $ \cu_{t\ge 0}$ as a full subcategory of $\cu$), and the composition $t_{\ge 0}=[1]\circ R_t\circ [-1]$ is left adjoint to  the embedding $ \cu_{t\le 0}\to \cu$.


\item\label{itha}
$\hrt$ is 
  an abelian category with short exact sequences corresponding to distinguished triangles in $\cu$.

\item\label{itho}
For any $n\in \z$ we will use the notation $t_{\ge n}$ for the functor 
 $[n]\circ L_t\circ [-n]$, and $t_{\le n}=[n+1]\circ R_t\circ [-1-n]$.

Then there is a canonical isomorphism of functors $t_{\le 0}\circ t_{\ge 0}\cong t_{\ge 0}\circ  t_{\le 0}$. 
  (if we consider these functors as endofunctors of $\cu$), and the composite functor $H^t=H_0^t$ actually takes values in the subcategory $\hrt$ of $ \cu$. Furthermore, this functor $H^t:\cu \to \hrt$   is homological.  
	
	\item\label{itperp} $\cu_{t\le 0}= \cu_{t\ge 1}\perpp$ and $\cu_{t\ge 0}=(\cu_{t\le- 1}^{\perp})$; hence these classes are retraction-closed and extension-closed in $\cu$.
	
	Moreover, if $t$ is generated by $\cp$ then $\cp\subset \cu_{t\ge 0}$.

\end{enumerate} 
\end{pr}
\begin{proof}
All of these statements were essentially established in \S1.3 of \cite{bbd}. 
\end{proof}

Now we are able to finish the proof of Proposition \ref{prmot}.

\begin{proof}[The proof of  Proposition \ref{prmot}(\ref{imotet}).]

Proposition \ref{ptstr}(\ref{itho},\ref{itperp}) easily implies that for  $\rgm'=\thomr{}_{\ge -2}\rgm$ the restrictions of the functors  $\dmer(\mgr(-),\rgm[2])$ and $\dmer(\mgr(-),\rgm'[2])$ to $\sm'$ are isomorphic indeed.
  Next,   Lemma 5.2 of \cite{kahnr} yields a description $\rgm$ as the corresponding total direct image of sheaves with respect to a morphism $\al$; it follows that   $H_{i}^{\thomr}(\rgm)=0$ for $i>0$ and $ H_{0}^{\thomr}(\rgm)=G_m\otimes R$ (see loc. cit). Loc. cit. also implies that  $ H_{-1}^{\thomr}(\rgm)=0$ (this is a form of Hilbert's theorem 90). 
Applying assertion Proposition \ref{prmot}(\ref{imot9}) to the long exact sequence coming from the corresponding $\thomr$-truncation of $\rgm'$ we obtain  for any $X\in \sv$ an 
 exact sequence as follows: $$H^2_{Nis}(X,G_m)\to \dmer(\mgr(X),\rgm'[2])(=
\bra(X))\to \for(S^{\bra})(X)\to H^3_{Nis}(X,G_m).$$ Since the groups $H^2_{Nis}(X,G_m)$ and $H^3_{Nis}(X,G_m)$ are well-known to be zero (see (6.2.2) and (6.4.2) of \cite{bev}), we obtain that $\for(S^{\bra})=\bra$ indeed.
\end{proof}

\begin{rema}\label{rz}
 1. Let us now describe a possible approach for studying unramified cohomology in the case where 
  $1/p\notin R$ (recall that this also means $p>0$); certainly, one may just take $R=\z$. Firstly, we note that the arguments relying on \cite{kabir} extend to this setting without any problem.


Next, we recall that the aforementioned weight structure $\wgbir$ on $\dmgbir$ exists in this case as well.

 Moreover, a reasonable candidate for the natural extension  $\wchows$ of $\wchowsg$ to $\dmger$ was proposed and studied in  \cite{bokum}. The main disadvantages of $\wchows$ is that we have no explicit description of 
 its heart, and we cannot prove that  
 it restricts to $\dmger$ (if $1/p\notin R$).

2. Recall that $\dmer\supset\dmger$ is a symmetric monoidal category as well. In Theorem 3.1.1(1) of \cite{bokum} it was proved that $\wchows$ restricts to the subcategory $\dmer\lan 1\ra$ of $\dmer$. In Proposition 3.2.1 of ibid. it was  deduced that there exists a weight-exact localization functor $\dmer\to \dmbir=\dmer/\dmer\lan 1\ra$. Moreover, the corresponding weight structure $\wbir$ is an extension of $\wgbir$; see Remark 3.2.2(1) of ibid. 

For this reason we formulate some more properties of  weight-exact localizations in Proposition \ref{plocw} below; note 
   also that part 1 of that proposition can be vastly generalized (see  \S3 of  \cite{bsnew}).

3. Another 
 relevant observation is that Theorem 3.2.3(I) of \cite{bvtr} easily yields the existence 
 a $t$-structure $\tcho$ on $\dmer$ that is {\it right adjacent} to $\wchows$ (i.e., $\dmer_{\tcho\ge 0}=\dmer_{\wchows\ge 0}$); see Remark 3.2.4(1) of ibid. The author is going to study the relation of $t$-structures obtained this way to unramified cohomology in detail in a forthcoming paper.
\end{rema}

\begin{pr}\label{plocw}
 Assume that $\cu$ is endowed with a weight structure $w$ that restricts to a triangulated subcategory $\du\subset \cu$,  and $\pi$ is the localization functor $\pi:\cu\to \cu'=\cu/\du$.   

\begin{enumerate}

\item\label{iwl2} Then there exists a weight structure  $w'$   on $\cu'$ such that the localization functor $\pi:\cu\to \cu'$ is weight-exact with respect to $(w,w')$. Moreover,    $\cu'_{w'\le 0}=\kar_{\cu'}(\pi(\cu_{w\le 0}))$,  $\cu'_{w'\ge 0}=\kar_{\cu'}(\pi(\cu_{w\ge 0}))$, and $\cu'_{w'= 0}=\kar_{\cu'}(\pi(\cu_{w= 0}))$.

Furthermore, the obvious functor $\frac \hw{\hw_{\du}}\to \hw'$ induced by $\pi$ is a full embedding; here $\hw_{\du}$ is the heart of the restriction of $w$ to $\du$, whereas $\frac \hw{\hw_{\du}}$ is the category whose object class equals $\obj \hw$ and  $\frac  \hw{\hw_{\du}} (X,Y)= \hw(X,Y)/(\sum_{Z\in \du_{w_{\du}=0}}\hw(Z,Y) \circ \hw(X,Z))$ for any $X,Y\in \obj \hw$.

\item\label{iwl312} 
 Assume that $N\in \cu'_{w'=0}$ and there exist $M',M\in \cu_{w=0}$ such that $N\bigoplus \pi(M')\cong \pi(M)$. Then there exists $M''\in \cu_{w=0}$ 
 such that $\pi(M'')\cong N$.

  Moreover, if  $\hw'$ along with the  class $\pi(\cu_{w= 0})\subset \cu'_{w'= 0}= \obj \hw'$   are closed with respect to countable $\hw'$-coproducts,  then $\pi(\cu_{w= 0})$   essentially equals  $\cu'_{w'= 0}$.
\end{enumerate}
\end{pr}
\begin{proof}

\ref{iwl2}. See Theorem 3.1.3(2),  Proposition 3.1.1(1,2), Theorem 3.2.2(3), and Remark 4.2.3 of \cite{bsnew}; note also that the first and the last statement in the assertion are given by Proposition 8.1.1 of \cite{bws}.	

\ref{iwl312}. 	  If  $N\bigoplus \pi(M')\cong \pi(M)$ then we can lift the corresponding monomorphism $i:\pi(M')\to \pi(M)$ and the projection $p':\pi(M)\to \pi(M')$ to $\cu$-morphisms $M'\stackrel{i}{\to}M \stackrel{p}{\to}M'$; see the previous assertion. 
 Moreover, since the restriction of $\pi$ to $\hw$ is the composition of the "factorization" of $\hw$ by $\hw_{\du}$ with a full embedding, the morphism $\id_{M'}-p\circ i$ factors through an object $D$ of $\hw_{\du}$. Hence $M'$ is a retract of $M\bigoplus D$, and applying 
 Proposition \ref{pbw}(\ref{itstable}) we obtain the following: there exists $M''\in \cu_{w=0}$ such that $M'\bigoplus M''\cong M\bigoplus D$ and the composition of the corresponding morphisms $M\to M'\to M$ equals $i\circ p$. Thus we can apply $\pi$ to the corresponding direct sum diagram to obtain that  $\pi(M')\bigoplus \pi(M'')\cong \pi(M)$  and the composition of the corresponding morphisms $\pi(M)\to \pi(M'')\to \pi(M)$ equals $\id_M-i'\circ p'$. Thus $\pi(M'')\cong N$ indeed.

Now recall that $\hw'$ is weakly idempotent complete by  Proposition \ref{pbw}(\ref{itstable}).  
  Hence assertion \ref{iwl2} implies that for any  $N_0\in \cu'_{w'=0}$ there exist $N'_0\in \cu'_{w'=0}$ and $M_0\in \cu_{w=0}$ such that $\pi(M_0)\cong N_0\bigoplus N'_0$.
	Thus it remains to verify that if   $\hw'$ and $\pi(\cu_{w= 0})$   are closed with respect to countable $\hw'$-coproducts then we can take $N_0'$ to belong to $ \pi(\cu_{w= 0})$. For thus purpose we apply 	the following well-known Eilenberg swindle argument: 
 we have $ \coprod_{\z}\pi(M_0)\cong \coprod_{\z} N_0 \bigoplus \coprod_{\z} N_0'\cong  N_0 \bigoplus \coprod_{\z} N_0 \bigoplus \coprod_{\z} N_0'\cong  N_0 \bigoplus \coprod_{\z}\pi(M_0)$; see the main result (Proposition) of \cite{freydsplit} for more detail.
\end{proof}

\begin{rema}\label{rsos}
1. The author is deeply grateful to Vladimir Sosnilo for 
 telling him the  main part of  Proposition \ref{plocw}(\ref{iwl312}).

2. Possibly, the 
 argument  above can be used for certain explicit calculations in the case where a decomposition of the sort $N\bigoplus \pi(M')\cong \pi(M)$ is known. 

3. It could make sense to apply the "weight lifting" statements of \cite{bsosnl} to motives. In particular, note that Theorem 3.3.1 of ibid. easily implies that in the setting of Proposition \ref{ploc} 
 the class $\cu'_{w'\ge 0}$ (resp.   $\cu'_{w'\le 0}$) essentially equals  $\pi(\cu_{w\ge 0})$ (resp. $\pi(\cu_{w\le 0})$); here one should also invoke Remark 3.3.2 of ibid. and categorical duality.
\end{rema}

\subsection{On weak weight complexes and their relation to Grothendieck groups}\label{skz}

Now let us discuss a certain generalization of  Theorem \ref{teuler}. To drop the assumption that $\cu$ possesses an $\infty$-enhancement (and $w$ is bounded) 
 we recall some of the theory of  so-called weak weight complex functors.

\begin{defi}\label{dbacksim}
Let $m_1,m_2:A\to B$ be  
morphisms of $\bu$-complexes, where $\bu$ is an additive category. Then 
we will write $m_1\backsim m_2$ if $m_1-m_2=d_Bh+jd_A$ for some collections of arrows $j^*,h^*:A^*\to B^{*-1}$.

We will call this relation the {\it weak homotopy one}. \end{defi}

To omit other technical definitions needed for the theory of weight complexes (cf. \S1.3 of \cite{bwcp}), we  will formulate the main properties of weight complex functors somewhat axiomatically.

\begin{pr}\label{pwtw}
I. Let $\bu$ be an additive category. 
\begin{enumerate}
\item\label{iwhecat}
Then factoring morphisms in $K(\bu)$ by the weak homotopy relation yields an additive category $\kw(\bu)$. Moreover, the corresponding full functor $K(\bu)\to \kw(\bu)$ is (additive and) conservative.

\item\label{iwhefu}
Let $\ca:\bu\to \au$ be an additive functor, where $\au$ is an abelian category. Then for any $B,B'\in \obj K(\bu)$ any pair of weakly homotopic morphisms $m_1,m_2\in C(\hw)(B,B')$  induce equal morphisms of the homology $H_*((\ca(B^i)))\to H_*((\ca(B'^i)))$.

Hence the correspondence $N\mapsto H_0(\ca(N^i))$ gives a well-defined functor $\kw(\bu)\to \au$.

\item\label{iwhefun} Applying an additive functor  $F:\bu\to \bu'$ to complexes termwise one obtains an additive functor $\kw(F):\kw(\bu)\to \kw(\bu')$.

\end{enumerate}
II. Assume that $\cu$ is endowed with a weight structure $w$.

Then there exists an additive functor $t:\cu\to \kw(\hw)$ that enjoys the following properties. 

\begin{enumerate}
\item\label{iwcbasew} The composition of the embedding $\hw\to \cu$ with $t$ is isomorphic to the obvious embedding $\hw \to \kw(\hw)$.

\item\label{irwcshw} Let $n\in \z$. Then $t\circ [n]_{\cu}\cong [n]_{\kw(\hw)}\circ t$, where  $[n]_{\kw(\hw)}$ is the obvious shift by $[n]$ (invertible) endofunctor of the category $\kw(\hw)$.

\item\label{iwconsw} If $M\in \cu_{w\le n}$ (resp. $M\in \cu_{w\ge n}$) then $t(M)$ belongs to $K(\hw)_{\wstu\le n}$ (resp. to $K(\hw)_{\wstu\ge n}$). 

\item\label{iwcexw} 
Let $M,M'\in \obj \cu$, $g\in \cu(M,M')$ (where $\cu$ is endowed with a weight structure $w$), and $h:M'\to \co(g)$ is the second side of a distinguished triangle containing $g$. 

Then there exists a lift of the $\kw(\hw)$-morphism chain $t(M)\stackrel{t(g)}{\to} t(M') \stackrel{t(h)}{\to} t(\co(g))$ to two sides of a distinguished triangle in $K(\hw)$.

\item\label{iwcfunctw} Let $\cu'$ be a triangulated category endowed with a weight structure $w'$; let
 $F:\cu\to \cu'$ be a weight-exact functor. Then the composition $t'\circ F$ is isomorphic to $\kw(\hf)\circ t$, where 
$t'$ is a weight complex functor corresponding to $w'$ and $\hf:\hw\to \hw'$ if the restriction of $F$ to hearts (see also assertion I.\ref{iwhefun} for the definition of $\kw(-)$). 

\item\label{iwcpurew} Let $\ca:\hw\to \au$ be an additive functor, where $\au$ is an abelian category. 
Then 
the composition $H^{\ca}=H_0\circ\kw(\ca)\circ t$  (that sends an object $M$ of $\cu$ into the zeroth homology of the complex $\ca(M^{j})$, where $(M^j)=t(M)$) is a homological functor. 

Moreover, if $w$ is bounded then this is the only homological functor (up to a unique isomorphism) whose restriction to $\hw$  equals $\ca$ and whose  restrictions to $\hw[i]$ for $i\neq 0$ vanish.

\end{enumerate}
\end{pr}
\begin{proof}
I. These statements easily follow from the results of \cite[\S3.1]{bws}; see Lemma 1.3.2 of \cite{bkw}. 

II. All these statements 
  are stated in \cite{bwcp}; see Proposition 1.3.4(6,7,8,12) 
	 and Theorem 2.1.2 of ibid.
\end{proof}

\begin{rema}\label{rwcw}
\begin{enumerate}
\item\label{irwc7} 
 $t$ can 
be "enhanced" to an  exact functor $t^{st}:\cu\to K(\hw)$ at least in the case where $\cu$ possesses an $\infty$-category model and either $w$ is bounded (see Proposition \ref{pwt} for this case) or $w$ is  {\it purely compactly generated} in the sense of  \cite[Remark 3.2.3(4)]{bwcp}; see Corollary 3.5 and Remark 3.6 of \cite{sosnwc}. 

Thus the functor $t^{st}$ exists for all "motivic" weight structures whose hearts consist of certain Chow motives. On the other hand, it is  currently not clear whether $t^{st}$ exists for the weight structure $\wchows$  mentioned in  Remark \ref{rz} (in the case $1/p\notin R$). 

\item\label{irwcu} 
In \S1.3 of \cite{bwcp} 
a  (weak) weight complex functor was 
 defined as a canonical additive functor $\cuw\to \kw(\hw)$, where $\cuw$ is a category canonically equivalent to $\cu$. Hence to define a functor $t$ as in Proposition \ref{pwtw} one should choose a splitting $\cu\to \cuw$ for this equivalence. Applying the conservativity of the projection $K(\hw)\to \kw(\hw)$ we obtain that for any object $M$ of $\cu$ the $K(\hw)$-isomorphism class of $t(M)$ does not depend on any choices. 

\item\label{irwc2} The weak homotopy equivalence relation was introduced  in \S3.1 of \cite{bws} independently from the earlier and closely related notion of {\it absolute homology}; cf. Theorem 2.1 of \cite{barrabs}. 
\end{enumerate}
\end{rema}

Now we are able to generalize Theorem \ref{teuler} significantly (recall that the corresponding $K_0$-groups are defined in Definition \ref{dk0}); certainly, the proofs have much in common. 
One of the most important cases is $\cu'=\cu$ and $\bu=\hw$; yet cf.  Remark \ref{reulerw}(3) below.

\begin{theo}\label{teulerw}
Assume that $\cu$ is endowed with a weight structure $w$,   $\cu'$ is a triangulated subcategory of $\cu$,  $\bu$ is an additive subcategory of $\hw$,  $N$ is an object of $\cu'$, $F:\bu\to \au$ is an additive functor (and $\au$ is additive), and   for any object $M$ of $\cu'$ the complex $t(M)$ is homotopy equivalent to an object $(M^i)$ of $K(\bu)$ such that $F(M^i)=0$ for almost all $i\in \z$.

I. 
 1. Then the correspondence $M\mapsto \sum_{i\in \z} (-1)^i[F(M^i)]$ gives a well-defined homomorphism $F_{K_0}:K_0^{\tr}(\cu')\to K_0^{\add}(\au)$. 

2. Assume that $\bu=\hw\subset \cu'$, $\au$ is an abelian category, and there exists a $\cu'$-morphism $h$ either from $N'$ to  $N$ or vice versa such that $F_{K_0}([\co(h)])=0$ and $N'\in \cu_{w=0}$. Then $F_{K_0}([N])=[F(N')]$. 

Moreover, if $H_j^F(\co(h))=0$ for 
$j=0,1$, then 
$F_{K_0}([N])=[H^{F}(N)]$; here  the homological functor $H^F$ is the one given by  Proposition \ref{pwtw}(II.\ref{iwcpurew}). 

3. Assume that $\au$ is abelian semi-simple. Then $F_{K_0}(N)=\sum_{i\in \z} (-1)^i[H^F_i(N)]$.

II. Assume that $\cu'=\cu$ and  $\bu$ 
strongly generates $\cu$.

1. Then the corresponding homomorphism $\id_{\bu,K_0}: K_0^{\tr}(\cu)\to K_0^{\add}(\bu)$ is an isomorphism.

In particular, this is the case if 
 $\bu=\hw$. 

2. Assume that $\bu=\hw$, $\au$ is abelian, $F=F'\circ \underline{HG}\circ i_w$, 
 where $i_w$ is the embedding $\hw\to \cu$, $\underline{HG}:\hw\to \hv$ is the restriction of a weight-exact functor $G:(\cu,w)\to (\du,v)$ to hearts, $F':\hv\to \au$ is an additive functor,  and for some $j\in \z$ 
 the object $G(N)$ belongs to $\du_{w=j}$. Then we have $F_{K_0}([N])=(-1)^j[H^{F}_j(N)]$. 
\end{theo}
\begin{proof}
I.1. First we check that our correspondence 
 gives a 
 function $\obj \cu'\to K_0^{\add}(\au)$ whose values on isomorphic objects are equal. For this purpose it suffices to verify for $\bu$-complexes $(M^i)\cong_{\kw(\hw)} (N^i)$ that $\sum_{i\in \z} (-1)^i[F(M^i)]=\sum_{i\in \z} (-1)^i[F(N^i)]$ if almost all of the summands in both parts are zero. Hence for the (contractible) cone $C=(C^i)$ of the corresponding $K(\bu)$-isomorphism (see Remark \ref{rwc}(\ref{irwcu})) we should prove that $\sum_{i\in \z} (-1)^i[F(C^i)]=0$. 

Note now that $F$ canonically extends to a functor $\kar(F):\kar(\bu)\to \kar(\au)$. Next, Proposition 10.9 of \cite{buhler} 
  says that $C$ splits as a $\kar(\bu)$-complex, i.e., $C$ is $C(\kar(\bu))$-isomorphic to $\bigoplus \id_{O^i}[-i]$ for certain $O^i\in \obj \kar(\bu)$. Our assumptions obviously imply the existence of $N_0>0$ such  that $\kar(F)(O^i)=0$ if $i<-N_0$ or $i>N_0$. Hence  $$\bigoplus_{-j\le i\le j}F(C^{2i})\cong  \bigoplus_{-j\le i\le j}F(C^{2i+1})$$ for $j\gg 0$, and we obtain the existence and the $\cu$-isomorphism invariance  of the function in question. 

Lastly, Proposition \ref{pwtw}(I.\ref{iwhecat}, II.\ref{iwcexw}) implies that for any $\cu'$-morphism $M\to M'$ there exists a $K(\hw)$-distinguished triangle $(M^i)\to (M'{}^i)\to t(\co(f))\to (M^i)[1]$, where $(M^i)$ and $(M'{}^i)$ are any  $\bu$-complexes that are isomorphic to $t(M)$ and $t(M')$, respectively. The existence of    (a group homomorpism) $F_{K_0}$ 
 follows immediately.

2. Immediately  from the definition of  $K_0^{\tr}(\cu')$, we have $F_{K_0}([N])=F_{K_0}([N'])$. Next, ("the correctness part" of) assertion I.1 implies that $F_{K_0}([N])=[H^F(N')]=F_{K_0}([N'])$.

Lastly, if $H^F(\co(h))=0=H^F(\co(h)[-1])$ then we have 
 $H^F(N')=H^F(N)$,  since the functor $H^F$ is homological.


3. This statement is an immediate consequence of our definitions. 

II.1.  
 The embedding $\bu\to \cu$ obviously gives a homomorphism  $K_0^{\add}(\bu)\to K_0^{\tr}(\cu)$; it is surjective since $\bu$ strongly generates $\cu$.  Hence $\id_{\bu,K_0}$ is an isomorphism, since its  composition with 
  $F_{\id}$ obviously equals $\id_{K_0^{\add}(\bu)}$. 
	
	Lastly, one can apply this statement if 
	  $\bu=\hw$  according to  Proposition \ref{pbw}(\ref{igenlm}); note here that $w$ is bounded according to Proposition \ref{pbw}(\ref{igen}).  

2. 
Proposition \ref{pwtw}(II.\ref{iwcfunctw}) easily implies that $F_{K_0}(N)=F'_{K_0}([G(N)])$. Hence we can assume that $\du=\cu$, $v=w$,  $G$ is the identity on $\cu$, $F=F'$, and $N\in \cu_{w=j}$. Then $N'=N[-j]\in \cu_{w=0}$; hence 
 for $t(N)=(N^i)$ the complex $(F(N^i))$ is $K(\au)$-isomorphic to the one-term complex $(F(N'))[j]$. 
 The  result in question follows immediately; cf. the proof of assertion I.2. 
\end{proof}

\begin{rema}\label{reulerw}
1. Part II.1 of our theorem generalizes Theorem 5.3.1 of \cite{bws}, where it was assumed that  $\bu$ is idempotent complete and equals $\hw$. 

Moreover, our theorem 
 clearly generalizes Theorem 1 of \cite{rose} where $\cu=K^b(\bu)$ was considered. On the other hand,  
 our arguments essentially "involve" Proposition 1 of ibid., and the formulation was partially inspired by ibid.

2. Obviously, 
  the homomorphism $F_{K_0}$ naturally extends to the triangulated subcategory of $\cu$ 
	 strongly generated by $\obj \cu'\cup \obj \bu$. 

 Next, if $\bu\subset \cu'$ then for any object $M$ of $\cu'$ that belongs to its subcategory strongly generated by $\bu$ and any homomorphism $h$ from $K_0^{\tr}(\cu')$ into an abelian group the 
 element  $h(M)$ is clearly canonically determined by the values of $h$ at the set $\{[M],\ M\in \obj \bu\}\subset K_0(\cu')$. 

In particular, the isomorphism provided by part II.1 of our theorem is canonical. 

3. Our motivation for considering the case $\cu'\neq \cu$ comes from   \cite{bokum}; see Remark \ref{rz}. In this context one can take $\cu=\dmer$, but $\cu'$ should be smaller since $K_0^{\tr}(\dmer)$ is easily seen to be zero.  

4. Probably, one can apply Theorem \ref{teuler}(II.2) to obtain a proof of Theorem \ref{tceked} that will not depend on Theorem \ref{tmotw}(\ref{imotp2}). For this purpose one may take $G$ to be the natural functor $\dmger\to \dmgbir$.
\end{rema}

\end{document}